\documentclass[11pt]{article}
\usepackage{amsmath,amsfonts,amsthm,amssymb,mathrsfs,amscd,latexsym,graphicx,colordvi}
\usepackage{color}
\allowdisplaybreaks

\newtheorem{theorem}{Theorem}
 \newtheorem{proposition}{Proposition}
 
 \newtheorem{remark}{Remark}
 \newtheorem{definition}{Definition}
 \newtheorem{lemma}{Lemma}
 \newtheorem{corollary}{Corollary}
 \newtheorem{example}{Example}

 \def \dist{{\rm dist} }

 \renewcommand{\Re}{\mathbb{R}}
  
  \newcommand{\V}{\mathcal{V}}

\newcommand{\R}{\mathbb{R}}
\newcommand{\ri}{{\rm ri\,}}

\newcommand{\bd}{{\rm bd\,}}
\newcommand{\inn}{{\rm int\,}}
\newcommand{\cl}{{\rm cl\,}}
\newcommand{\co}{{\rm co\,}}

\newcommand{\norm}[1]{\|#1\|}

\date{}

\begin{document}

\title{
Second-order  optimality conditions for non-convex  set-constrained optimization problems}

\author{Helmut Gfrerer\thanks{Institute of Computational Mathematics, Johannes Kepler University Linz, A-4040 Linz, Austria, e-mail: helmut.gfrerer@jku.at. This author's research was supported by the Austrian Science Fund (FWF) under grant P29190-N32.} \ \ \ \ \ \ \ \ Jane J. Ye\thanks{Department of Mathematics
and Statistics, University of Victoria, Victoria, B.C., Canada V8W 2Y2, e-mail: janeye@uvic.ca. The research of this author was partially
supported by NSERC.}\ \ \ \ \ \ \ \ Jinchuan Zhou\thanks{Department of Statistics, School of Mathematics and Statistics, Shandong University of Technology,
 Zibo 255049, P.R. China, e-mail: jinchuanzhou@163.com. This author's work is supported by National Natural Science Foundation of China (11771255, 11801325) and Young Innovation Teams of Shandong Province (2019KJI013).}}

 \maketitle

 \noindent
 {\bf Abstract.}\ In this paper we study second-order  optimality conditions for non-convex set-constrained optimization problems.  For
 a convex set-constrained optimization problem,  it is well-known that second-order optimality conditions involve the support function of the second-order tangent set.
 In this paper we propose two approaches for establishing second-order optimality conditions for the non-convex case. In the first approach we extend the concept of the support function so that it is applicable to general non-convex set-constrained problems, whereas in the second approach we introduce the notion of the directional regular tangent cone and apply classical results of convex duality theory. Besides the second-order optimality conditions, the novelty of our approach lies in the systematic introduction and use, respectively, of directional versions of well-known concepts from variational analysis.
 %In this paper we extend the concept of the support function to allow  nonconvexity and  establish some second-order  optimality conditions for this very general optimization problem. For %the classical case of convex set-constrained problem, our new  necessary optimality conditions  hold under  weaker constraint qualifications.
%  by using the fact that the MSCQ has inherited  property, i.e., the first-order and second-order linearized subproblems has MSCQ provided that the primal problem has MSCQ.



\newpage
 \section{Introduction}

 Second-order optimality conditions have long been recognized as an  important tool  in  optimization theory and algorithms.
 %Among other usages, second-order necessary optimality conditions can be used to distinguish a local optimal solution from other non-optimal stationary points.
In this paper we aim at developing second-order
 optimality conditions for   a  set-constrained optimization problem in the form of
 \begin{flalign}\label{primal problem}
(P)\hspace*{5cm} &\min\ \  f(x) \ \ \  {\rm s.t.}\ \ g(x)\in \Lambda, &
 \end{flalign}
 where $f:\Re^n\rightarrow  \Re$ and $g :\Re^n\rightarrow  \Re^m$  are twice continuously differentiable, and $\Lambda$ is a closed subset in $\Re^m$.  For the case where $\Lambda$ is convex, a complete theory of  second-order necessary and sufficient optimality conditions  has been developed by Bonnans, Cominetti and Shapiro in  \cite{BCS99} and the results have been reviewed in the monograph of Bonnans and Shapiro \cite{BS}.

 In recent years, some important problem classes which can be reformulated in the form of problem (P) with non-convex $\Lambda$ have attracted much attention from the optimization community. These problems include the mathematical program with complementarity constraints (MPCC) (see e.g. \cite{lpr}),  the mathematical program with second-order cone complementarity constraints (SOC-MPCC) (see e.g. \cite{YZ15}) and  the mathematical program with
 semi-definite cone complementarity constraints (SDC-MPCC) (see e.g. \cite{DSY14}). Unlike the first-order optimality conditions for which much research works have been appeared, there %are
{is}  very little research done with the second-order optimality conditions for MPCC, SOC-MPCC and SDC-MPCC, let %along
alone
 the general non-convex set-constrained problem (\ref{primal problem}). The classical second-order necessary optimality condition for MPCCs was given  in \cite[Theorem 7(1)]{SS} under the MPCC strict Mangasarian-Fromovitz constraint qualification (SMFCQ).  Some weaker second-order necessary optimality conditions for MPCCs were derived in \cite{GGY}.  {For the case when $\Lambda$ is the union of finitely many convex polyhedral sets, which comprises MPCCs, strong second-order necessary optimality conditions were given in \cite{Gfr14a} under a directional metric subregularity constraint qualification which is much weaker than SMFCQ.}
 Recently a second-order necessary optimality condition is derived in  \cite[Theorem 5.1]{CYZZ} for SOC-MPCCs under the nondegeneracy condition (equivalently the generalized linear independence constraint qualification, generalized LICQ).

 To our knowledge, there is no work dealing with the second-order optimality condition for the general non-convex set-constrained problem in the form (\ref{primal problem}).
The main purpose of this paper is to fill this gap.

In case of non-convex set  $\Lambda$,  first-order necessary optimality conditions can be derived by means of variational analysis. From the work of Bonnans, Cominetti and Shapiro \cite{BCS99},  it is well known that the second-order optimality condition must involve in some way the second-order tangent set to the set $\Lambda$, and the non-convexity of this second-order tangent set is also an issue.  We consider two different approaches for handling non-convex  second-order tangent sets. In the first approach, solely based on non-convex variational analysis, we first show that directional metric subregularity of the feasible set mapping carries over to the  second-order linearized subproblem. Then we introduce the concept of {\em lower generalized support function} (which coincides with  the support function in the convex case) in order to state the second-order necessary optimality conditions in Theorem \ref{second-order-theorem-1}.

In the second approach, convex duality plays an essential role. We first introduce the {\em directional regular (Clarke) tangent cone} to $\Lambda$ and state its relations with the directional limiting normal cone and the second-order tangent set. Using these relations we {introduce a new constraint qualification called {\em directional Robinson's constraint qualification} and we} can carry over the ideas already employed in \cite{BCS99} to obtain the second-order necessary optimality conditions of Corollary \ref{CorClSecondOrder}.
 We show that these second-order conditions are equivalent with primal second-order conditions and are in general stronger than the one obtained with our first approach. However, they also require a stronger constraint qualification and their practical use is limited by the fact that we have not only one condition for every critical direction, but for every convex set contained in the second-order tangent set. If we further strengthen the constraint qualification to some {\em directional non-degeneracy condition},  this drawback vanishes and we can state a single condition involving the support function of the second-order tangent set.
 
Second-order optimality  {conditions} have in general the form that for every critical direction some  {conditions are} fulfilled. It seems to be that Penot \cite{Pen98} was the first who recognized that only some directional form of a constraint qualification (directional metric subregularity) is required for stating the necessary conditions. We pursue this approach and, as a byproduct of the second-order optimality conditions, we introduce and analyze a lot of new directional objects like the directional regular tangent cone, the directional Robinson's constraint qualification and directional non-degeneracy. These results are of its own interest.

We organize our paper as follows. Section 2 contains the preliminaries and preliminary results. In Sections 3 and 4, we derive the primal and dual form of second-order necessary optimality conditions, respectively. Section 5 discusses second-order sufficient conditions for optimality. In Section 6 we present four examples which illustrate our second-order necessary and sufficient conditions.
% In Section 6
%we apply our results to the nonlinear programming, the second-order cone programming and use an example to illustrate our results.

 \section{Preliminaries and preliminary results}
In this section we  clarify the notations,  recall some background material we need from variational analysis and
 develop some preliminary results.

% We denote by $\Re_+$ and $\Re_{++}$ the set of nonnegative scalars and
% positive scalars respectively, i.e., $\Re_+:=\{\alpha|\ \alpha\geq 0\}$ and
% $\Re_{++}:=\{\alpha|\ \alpha>0\}$.
 The  unit sphere in $\Re^n$ is  denoted by $\mathbb{S}$ and the open and closed unit balls  are denoted by $\mathbb{B}$ and ${\cal B}$ respectively. For a set ${ C}$,  denote by $\inn C$, $\ri C$, $\cl C$, $\bd C$, $\co C$, %$C^c$
{$C^{\rm comp}$ its interior, relative interior}, closure, boundary,  convex hull, and its complement, respectively. For a closed set $C\subseteq \Re^n$, let
$C^\circ$ and $\sigma_C(x)$ or $\sigma(x|C)$ stand for the polar cone and the support function of $C$, respectively, i.e.,
 %and dual cone of $A$ respectively,
 $C^\circ:=\{v|\ \langle v,w \rangle\leq 0,\  \forall w\in C \}$
 and $\sigma_C(x)=\sigma(x|C):=\sup\{\langle x,x' \rangle|x'\in C\}$ for $x\in \Re^n$.  {Let  $ o(\lambda):\mathbb{R}_+\rightarrow \mathbb{R}^m$ stand for} a mapping with the property that $o(\lambda)/\lambda\rightarrow 0$ when $\lambda \downarrow 0$.  $z\overset{S}{\to} x$ means $z\in S$ and $z\rightarrow x$.
 % For a differentiable  mapping $\Phi:\mathbb{R}^n\to \mathbb{R}^m$ and a vector $x\in \mathbb{R}^n$, we denote by
 %$D \Phi(x)$ the Jacobian matrix of $\Phi$ at $x$, and
  %$\nabla \Phi(x)\in \mathbb{R}^{m\times n}$ the Jacobian matrix of $\Phi$ at $x$.
  For a   mapping $\Phi:\mathbb{R}^n\to \mathbb{R}^m$,  we denote by
 $\nabla \Phi(x)\in \mathbb{R}^{m\times n}$ the Jacobian and  by $\nabla^2 \Phi(x)$ the second order derivative as defined by
 $$u^T\nabla ^2 \Phi(x) :=\lim_{t\rightarrow 0} \frac{\nabla \Phi( x+ tu)-\nabla \Phi( x)}{t} \quad \forall u \in \mathbb{R}^n.$$
 Hence, for a scalar mapping $f:\mathbb{R}^n\to \mathbb{R}$, $\nabla^2 f(x)$ can be identified with the Hessian and for a mapping $\Phi:\mathbb{R}^n\to \mathbb{R}^m$, we have
 $$\nabla^2 \Phi(x)(d,d):=d^T \nabla^2 \Phi(x) d=(d^T \nabla^2 \Phi_1(x)d, \dots, d^T \nabla^2 \Phi_m(x)d)^T \quad \forall d\in \mathbb{R}^n.$$
 %Let $D^*F$ stand for the Mordukhovich coderivative of a
 %set-valued mapping $F$.

Let $\Phi: \mathbb{R}^n \rightrightarrows \mathbb{R}^m$ be a set-valued mapping. We denote by $\limsup_{x'\rightarrow x}\Phi(x')$ and $\liminf_{
x'\rightarrow x}\Phi(x')$ the
 Painlev\'{e}-Kuratowski upper and lower limit, i.e.,
\begin{eqnarray*}
& & \limsup_{x'\rightarrow x}\Phi(x) :=\left  \{v\in \mathbb{R}^m\Big| \exists x_k \rightarrow x, v_k\rightarrow v \mbox{ with } v_k\in \Phi(x_k) \right\}\\
   && \liminf_{x'\rightarrow x}\Phi(x):=\left  \{v\in \mathbb{R}^m\Big| \forall  x_k \rightarrow x, \exists v_k\rightarrow v \mbox{ with } v_k\in \Phi(x_k) \right\},
\end{eqnarray*}
respectively.

 \begin{definition}[Tangent cones] \cite[Definitions 2.54 and 3.28]{BR05} Given $S\subseteq \mathbb{R}^n$ and $x\in S$, the regular/Clarke and
 (Bouligand-Severi) {\em tangent/contingent cone} to $S$ at $x$ are defined respectively by
 \begin{eqnarray*}
  \widehat{T}_S(x)&:=& { \liminf\limits_{{x' \stackrel{S}{\to}x} \atop {t\downarrow 0}}\frac{S-x'}{t}=\Big\{ d\in \Re^n\,\Big|\, \forall \, t_k\downarrow 0,\, {x_k \stackrel{S}{\to}x}, \; \exists d_k\to d \
  \ {\rm with}\ \ x_k+t_kd_k\in S  \Big\}},\\
 T_S(x)&:=& \limsup\limits_{t\downarrow 0}\frac{S-x}{t}
 = \Big\{d\in\Re^n \, \Big| \, \exists \ t_k\downarrow 0,\;d_k\to d \ \ {\rm with}
 \ \ x+t_k d_k\in S \Big\}.
 \end{eqnarray*}
For $x\in S$ and $d\in T_S(x)$,
 %the inner and
the  outer second-order tangent set to $S$ in the direction $d$ is defined by
% respectively as
% \[
% T_S^{i,2}(x, d) = \left \{ w \in \Re^n \, \bigg | \,
% \dist \left( x+t d + \frac{1}{2}t^2w, S \right) = o(t^2) \right \}
% \]
% and
  \begin{eqnarray*}
 T_S^2(x; d) &:=&\limsup\limits_{t\downarrow 0}\frac{S-x-td}{\frac{1}{2}t^2}\\
& =& \left \{ w \in \Re^n \, | \, \exists \ t_k \downarrow 0, v_k \rightarrow w \
 {\rm such \ that}\  x+t_kd+\frac{1}{2}t^2_k v_k \in S
 %\dist \big( x+t_k d + \frac{1}{2}t_k^2w, S \big)= o(t_k^2)
  \right \}.
 \end{eqnarray*}
 \end{definition}

 Alternatively, the contingent  cone and the second-order tangent set  can be written in the form
 \begin{eqnarray}
 T_S(x)&=&  \Big\{d\in\Re^n \, \big| \, \exists \ t_k\downarrow 0, {\rm dist}(x+t_kd, S)=o(t_k) \Big\},\label{dist1}\\
 T_S^2(x;d)&=&   \Big\{{w}\in\Re^n \, \big| \, \exists \ t_k\downarrow 0, {\rm dist}(x+t_kd+\frac{1}{2}t_k^2 w, S)=o(t_k^2) \Big\}, \label{dist2}
 \end{eqnarray}
 respectively; see \cite[(2.87) and (3.50)]{BS}.
 The regular tangent cone is always a closed convex cone.  The tangent cone is always a closed cone and it is a closed convex cone provided that the set $S$ is convex. However the outer second-order tangent set  may be a non-convex set even when the set $S$ is convex (see \cite[Example 3.35]{BS}).  While the tangent cone contains zero always, the second-order tangent set may not be a cone and it may be empty (see e.g. \cite[Example 3.29]{BS}).

%The following results are from \cite[page 168]{BS}.
%\begin{proposition}  Given a convex set $S\subseteq \mathbb{R}^n$ and $d\in T_S(x)$, one always have
%\begin{equation}\label{tangtang} T_S^2(x;d)\subseteq T_{T_S(x)}(d)
%\end{equation}
%and the equality holds if $0\in T_S^2(x;d)$  (e.g. $S$ is a convex polyhedral set).
%\end{proposition}

We now introduce a concept of directional regular/Clarke tangent  cone which we will need later.
{The following definition is motivated by the formula
$\widehat T_S(x)=\displaystyle \liminf_{x' \overset{S}  \to x} T_S(x'),$
whenever $S$ is locally closed at $x$, cf. \cite[Theorem 6.26]{RW98}.}
 \begin{definition}[Directional regular/Clarke tangent cone] Given $S\subseteq \mathbb{R}^n$,  $x \in S$ and $d \in \mathbb{R}^n$, the regular/Clarke tangent cone to $S$ at $x$ in direction $d$ is defined by
 \begin{eqnarray*}
 \lefteqn{ \widehat{T}_{S}(x;d):=}
&&\qquad \qquad  \liminf\limits_{{t\downarrow 0, d'\to d}\atop {{x}+td'\in S}}T_S({x}+td')\\
&&\qquad =\Big\{v\in \Re^n \big| \, \forall t_k\downarrow 0, d_k \to d,   x+t_k d_k \in S, \exists v_k\rightarrow v \mbox{ with } v_k\in T_S( x+t_k d_k) \Big\}.
 \end{eqnarray*}
\end{definition}

It is easy to see from definition that for a closed set $S$ the directional version of the regular tangent cone contains the non-directional one and it coincides with the non-directional one when the direction is equal to zero, i.e., $\widehat{T}_{S}(x;d)\supseteq \widehat{T}_{S}(x)$  and $\widehat{T}_{S}(x;0)=\widehat{T}_{S}(x)$.

\if{
 \begin{proposition}\label{Prop2.2}
 Given a set $S\subseteq \mathbb{R}^n$ and $x\in S$, one has \begin{eqnarray*}T_S(x)+\widehat{T}_S(x)=T_S(x).\end{eqnarray*}
For any $x\in S$ and
$d\in T_S(x)$, one has
{$$T_S^2(x; d)+\widehat T_S(x;d) = T_S^2(x; d).$$}
\end{proposition}
\begin{proof}  We omit the proof for the first equation since the proof is similar.
{The inclusion $\supseteq$ in the second equation is clear, since $0\in \widehat{T}_S(x;d)$}. Conversely, let $ w\in T_S^2(x;d), v\in \widehat T_S(x;d)$ and we prove that $w+v \in T_S^2(x; d)$ by contradiction. To the contrary, suppose that $w+v \not \in T_S^2(x; d)$. Then there exists $\epsilon>0, t_k\downarrow 0,w_k\rightarrow w$ such that $x +t_k d+\frac{1}{2} t_k^2 w_k \in S$ and
$$ {\rm dist}(x +t_k d+\frac{1}{2} t_k^2 (w_k+v), S)\geq \frac{1}{2} t_k^2\epsilon .$$ Denote by  ${x}_k:=x +t_k d+\frac{1}{2} t_k^2 w_k, \tau_k:=\frac{1}{2} t_k^2$,
the above inequality is
%$$ {\rm dist} (x_k +\tau_k v, S) \geq \epsilon \tau_k,$$
 equivalent to saying that
$$ \left ( x_k +\tau_k(v+\epsilon \mathbb{B})\right )\cap S=\emptyset.$$ By the proof of \cite[Theorem 2.26]{RW98}, there exists $\tilde{x}_k\in S\cap  ( x_k +\tau_k(\|v\|+\epsilon )\mathbb{B})$ such that
\begin{eqnarray*}
%&& \|\tilde{x}_k-x_k\|\leq \tau_k (\|v\|+\epsilon),\\
&& {\rm dist} (v, T_S(\tilde{x}_k)) \geq \epsilon.
\end{eqnarray*}
By definition of the directional regular tangent cone, the above  implies that $v\not \in \widehat{T}_S(x;d)$. But this is a contradiction and so the proof is complete.
\end{proof}
}\fi
We now derive some properties of first and second-order tangent sets. The formula for the second-order tangent set extends the one in \cite[Proposition 13.12]{RW98}.
\begin{proposition}\label{PropRegTanCone}
 Given a closed set $S\subseteq \mathbb{R}^n$, for every $x\in S$ and every
$d\in T_S(x)$ one has
\[T_{T_S(x)}(d)+\widehat T_S(x;d)=T_{T_S(x)}(d),\quad T_S^2(x; d)+\widehat T_S(x;d) = T_S^2(x; d).\]
\end{proposition}
\begin{proof}
The inclusion $\supseteq$ in both equations is clear, since  $0\in \widehat{T}_S(x;d)$. In order to show the inclusion $\subseteq$ in the first equation, consider $w\in T_{T_S(x)}(d)$ and $v\in \widehat T_S(x;d)$ and we prove that $w+v \in T_{T_S(x)}(d)$ by contradiction. To the contrary, suppose that $w+v \not \in T_{T_S(x)}(d)$. Then by virtue of (\ref{dist1}) there is some $\epsilon>0$ and some $\bar t>0$ such that
\[\dist\big(d+t(w+v),T_S(x)\big)\geq 4\epsilon t\quad \forall t\in (0,\bar t).\]
Consequently, $d+t(w+v) \not \in T_S(x)$ and hence for every $t\in(0,\bar t)$ there is some $\bar\alpha_t>0$ with
\[\dist\big(x+\alpha(d+t(w+v)),S\big)\geq 3\epsilon\alpha t\quad \forall\alpha\in (0,\bar\alpha_t).\]
Since $w\in T_{T_S(x)}(d)$, by definition  there are sequences $t_k \downarrow 0$ and $w_k\to w$ such that $d+t_kw_k\in T_S(x)$ for all $k$. Thus, for every $k$ there is a sequence $\alpha^k_i\downarrow 0$ and {$d^k_i\to d$},
as $i\to\infty$ satisfying $x+\alpha^k_i({d^k_i+t_kw_k})\in S$ for all $i$. For every $k$ sufficiently large we have $t_k<\bar t$, $\norm{w_k-w}<\epsilon$ and we can find some index $i(k)$ such that $\alpha^k_{i(k)}<\min\{\frac 1k,\bar\alpha_{t_k}\}$ and $\norm{d^k_{i(k)}-d}<\epsilon t_k$. It follows together with Lipschitz property of the distance function that
\begin{eqnarray*}\lefteqn{\dist\big(x+\alpha^k_{i(k)}(d^k_{i(k)}+t_k(w_k+v)),S\big)}\\
&\geq& \dist\big(x+\alpha^k_{i(k)}(d+t_k(w+v)),S\big)-\alpha^k_{i(k)}(\norm{d^k_{i(k)}-d}+t_k\norm{w_k-w})\\
&>& 3\epsilon \alpha^k_{i(k)}t_k-2\epsilon \alpha^k_{i(k)}t_k=\epsilon\alpha^k_{i(k)}t_k\end{eqnarray*}
implying
\[ \left ( x_k +\tau_k(v+\epsilon \mathbb{B})\right )\cap S=\emptyset\]
with $x_k:=x+\alpha^k_{i(k)}(d^k_{i(k)}+t_kw_k)\in S$ and $\tau_k:=\alpha^k_{i(k)}t_k$. By the proof of \cite[Theorem 6.26]{RW98}, there exists $\tilde{x}_k\in S\cap  ( x_k +\tau_k(\|v\|+\epsilon )\mathbb{B})$ such that
\begin{eqnarray}
&& {\rm dist} (v, T_S(\tilde{x}_k)) \geq \epsilon.\label{dist3}
\end{eqnarray}
Since
\begin{eqnarray*}\norm{\tilde x_k-(x+\alpha^k_{i(k)}d)}&\leq& \norm{\tilde x_k-x_k}+\norm{x_k-(x+\alpha^k_{i(k)}d)}\\
&\leq& \alpha^k_{i(k)}\Big(t_k(\norm{v}+\epsilon)+\norm{d^k_{i(k)}-d}+t_k\norm{w_k}\Big)\\
&=&o(\alpha^k_{i(k)}),\end{eqnarray*}
we have $\tilde x_k=x+\alpha^k_{i(k)}d_k$ with some sequence $d_k \rightarrow d$.
Together with (\ref{dist3}) this  implies that $v\not \in \widehat{T}_S(x;d)$ which contradicts the assumption that $v \in \widehat{T}_S(x;d)$ and hence we have proved that $w+v\in T_{T_S(x)}(d)$. Indeed, to the contrary if $v \in \widehat{T}_S(x;d)$, then  by definition of the directional regular tangent cone, for the sequence $\alpha^k_{i(k)}\downarrow 0, d_k\rightarrow d$, $\tilde x_k\in S$, there must exists a sequence $v_k\rightarrow v$ with $v_k \in T_S(\tilde x_k)$, contradicting (\ref{dist3}).

We show the inclusion $\subseteq$ in the second equation in a similar way.
Let $ w\in T_S^2(x;d), v\in \widehat T_S(x;d)$ and we prove that $w+v \in T_S^2(x; d)$ by contradiction. To the contrary, suppose that $w+v \not \in T_S^2(x; d)$. Then by virtue of (\ref{dist2}) there exists $\epsilon>0, t_k\downarrow 0,w_k\rightarrow w$ such that $x +t_k d+\frac{1}{2} t_k^2 w_k \in S$ and
$$ {\rm dist}(x +t_k d+\frac{1}{2} t_k^2 (w_k+v), S)\geq \frac{1}{2} t_k^2\epsilon .$$
Denote by  ${x}_k:=x +t_k d+\frac{1}{2} t_k^2 w_k, \tau_k:=\frac{1}{2} t_k^2$,
the above inequality is  equivalent to saying that
$$ \left ( x_k +\tau_k(v+\epsilon \mathbb{B})\right )\cap S=\emptyset.$$
Now we can proceed similar as before to obtain the contradiction $v\not\in \widehat{T}_S(x;d)$.
\end{proof}

\begin{definition}[Normal Cones]  \label{NormalCone}(See e.g. \cite{Mor})  Given $S\subseteq \mathbb{R}^n$ and $x\in S$, the regular/Fr\'echet  normal cone to $S$ at $x$  is given by
 \[
 \widehat{N}_S(x):=\left\{v\in \Re^n \, \Big| \, \langle v, {x'}-x\rangle \le
 o\big(\|x'-x\|\big), \ \forall x'\in S\right\};
 \]
the  limiting/Mordukhovich normal cone to $S$ at $x$  is
 defined as
 \[
 N_S(x):=\limsup\limits_{x'\overset{S}{\to} x}\widehat{N}_S(x'),
 \]
and the Clarke normal cone to $S$ at $x$ is $N^c_S(x):=\cl\co N_S(x).$

 \end{definition}
 The limiting normal cone is in general non-convex whereas the Fr\'{e}chet normal cone is always convex. In the case of a convex set $S$, both the Fr\'{e}chet  normal cone and the limiting normal cone coincide with the normal cone in the sense of convex analysis, i.e.,  $$ N_S(x):=\left \{ v\in \Re^n \, \big | \, \langle v, {x'}-x\rangle \le
0, \  \forall x'\in S\right  \} .$$

Recently a directional version of limiting normal cones were introduced in \cite{GM} and extended to   general Banach spaces in \cite{Gfr13a}.
\begin{definition}[Directional Limiting Normal Cones]  Given a set
$S\subseteq\mathbb R^n$, a point $x\in S$
and  a direction $d\in \mathbb{R}^{n}$, the limiting normal  cone to $S$ in  direction $d$ at $x$ is defined by
\[
N_{S}(x; d):=\limsup\limits_{{t\downarrow 0, d'\to d}}\widehat{N}_S(x+td')=\left \{v | \exists t_{k}\downarrow 0, d_{k}\rightarrow d, v_{k}\rightarrow v \ {\rm with }\  v_{k}\in \widehat{N}_{S}(x+ t_{k}d_{k}) \right \}.
\]
\end{definition}
From definition, it is obvious that $N_{S}(x; d)=\emptyset$ if $d \not \in T_S(x)$,  $N_{S}(x;d)\subseteq N_S(x)$, and $N_{S}(x; 0)=N_S(x)$. {When $S$ is convex and $d\in T_S(x)$ there holds
\begin{equation}\label{EqDirLimNormalConeConvex}
  N_S(x; d)= N_S(x)\cap \{d\}^\perp = N_{T_S(x)}(d),
\end{equation}
cf. \cite[Lemma 2.1]{Gfr14a}.}
The following result is the directional counterpart of the fact that the limiting normal cone mapping is outer semicontinuous (see e.g.\cite[Proposition 6.6]{RW98}).
\begin{proposition}\label{LemRobustDirNormalCone}Given a set
$S\subseteq\mathbb R^n$, a point $x\in S$
and  a direction $d\in \mathbb{R}^{n}$, one has
  \[
N_{S}(x; d)=\limsup\limits_{{t\downarrow 0, d'\to d}}N_S(x+td')=\Big \{v | \exists t_{k}\downarrow 0, d_{k}\rightarrow d, v_{k}\rightarrow v \ {\rm with }\  v_{k}\in N_{S}(x+ t_{k}d_{k}) \Big \}.
\]
\end{proposition}
\begin{proof}
  The inclusion $N_{S}(x; d)\subseteq\limsup\limits_{{t\downarrow 0, d'\to d}}N_S(x+td')$ follows easily from the fact that for every $x'$ we have $\widehat N_S(x')\subseteq N_S(x')$. In order to show the reverse inclusion, consider sequences $t_k\downarrow 0$, $d_k\to d$ and $v_k\to v$ with $v_k\in N_S(x+t_kd_k)$. By the definition of limiting normals, for every $k$ there exist sequences $x_k^i\to x+t_kd_k$ and $v_k^i\to v_k$ as $i\to\infty$ with $v_k^i\in \widehat N_S(x_k^i)$. Using a standard diagonal process, for every $k$ we can find some index $i(k)$ satisfying
  \[\norm{x_k^{i(k)}-(x+t_kd_k)}\leq \frac{t_k}k,\ \norm{v_k^{i(k)}-v_k}\leq \frac 1k.\]
  Setting $d_k':= (x_k^{i(k)}-x)/t_k$, it follows that $\norm{d_k'-d_k}\leq \frac 1k$ and consequently $d_k'\to d$. Since $\lim_{k\to\infty} v_k^{i(k)}=\lim_{k\to\infty} v_k= v$ and $v_k^{i(k)}\in \widehat N_S(x+t_kd_k')$, $v\in N_S(x;d)$ follows. Hence, the inclusion $N_{S}(x; d)\supseteq\limsup\limits_{{t\downarrow 0, d'\to d}}N_S(x+td')$ is also established and the proof is complete.
\end{proof}

From the definition of the Clarke normal cone in Definition \ref{NormalCone}, it is natural to define the directional Clarke normal cone as follows.
\begin{definition}[Directional Clarke Normal Cones]  Given a set
$S\subseteq\mathbb R^n$, a point $x\in S$
and  a direction $d\in \mathbb{R}^{n}$, the Clarke normal  cone to $S$ in  direction $d$ at $x$ is defined by
$$N^c_S(x;d):=\cl\co N_S(x;d).$$
\end{definition}
Similarly to the directional limiting normal cone, we also have $N_{S}^c(x; d)=\emptyset$ if $d \not \in T_S(x)$,  $N_{S}^c(x;d)\subseteq N_S^c(x)$ and $N_{S}^c(x; 0)=N_S^c(x)$.

Similar to the standard tangent-normal polarity (see \cite[Theorem 6.28]{RW98}, \cite{Clarke1983}), we have the following directional tangent-normal polarity.
 \begin{proposition}[Directional Tangent-Normal Polarity]\label{polarity}
 For a closed set $S$, $ x\in S$, and $d\in \Re^n$, one has $$\widehat{T}_S({x};d)=N_S({x};d)^\circ=N_S^c({x};d)^\circ, \quad \widehat{T}_S({x};d)^\circ=N_S^c({x};d).$$
{In particular, the directional regular tangent cone $\widehat{T}_S({x};d)$ is closed and convex.}
 \end{proposition}

 \begin{proof}
 Firstly we show $\widehat{T}_S({x};d)\subseteq N_S({x};d)^\circ$. For any given $w\in \widehat{T}_S({x};d)$, take
 $v\in N_S({x};d)$. By the definition of directional normal cone, there exist $v_n\to v$ with
 $v_n\in \widehat{N}_S(x+t_nd_n)$ for some $t_n\downarrow 0$ and $d_n\to d$ and $x+t_nd_n\in S$. Since $w\in \widehat{T}_S({x};d)$, for this sequence, by the definition of directional regular tangent cone, there exists $w_n\to w$ with $w_n\in T_S(x+t_nd_n)$. It follows that
 $\langle w_n,v_n \rangle\leq 0$ since $v_n\in \widehat{N}_S(x+t_nd_n)=(T_S(x+t_nd_n))^\circ$. Taking the limit yields $\langle w,v \rangle\leq 0$, which implies that $w\in  N_S( x;d)^\circ$. Hence $\widehat{T}_S(x;d)\subseteq N_S(x;d)^\circ$.

 Secondly we  show $\widehat{T}_S(x;d)\supseteq N_S(x;d)^\circ$.
 % or equivalently $\widehat{T}_S(x;d)^c\subseteq (N_S(x;d)^\circ)^c$.
  Suppose that $w\notin \widehat{T}_S(x;d)$.
  %By definition of inner limits (see Page 152, variational analysis),
  Then there exist $t_n\downarrow 0$, $d_n\to d$, $x+t_nd_n\in S$ such that $w\notin \liminf\limits_{n\to \infty}T_{S}(x+t_nd_n)$. Hence by \cite[Exercise 4.2(a)]{RW98}, $\displaystyle \limsup_{n\rightarrow \infty} d(w, T_{S}(x+t_nd_n))>0$,
   i.e.,  there exists $\epsilon>0$ and some subsequence $ \mathcal{N}\subseteq \{1,2,\dots,\}$ such that
 ${\rm dist}(w, T_{S}(x+t_nd_n))>\epsilon$ for all $n\in  \mathcal{N}$.
According to \cite[Proposition 6.27(b)]{RW98}, there exists $v_n\in N_S(\bar{x}+t_nd_n)\cap \mathbb{S}$ such that
 \begin{equation}\label{com-3-1}
 \langle w,v_n \rangle {=} {\rm dist}(w, T_{S}(\bar{x}+t_nd_n))>\epsilon
 \end{equation}
 for all $n\in \mathcal{N}$. Since $v_n$ is bounded, we can assume by further taking subsequence if necessary that $v_n\to v$ (as $n\to \infty$ and $n\in\mathcal{N}$). Clearly by Proposition \ref{LemRobustDirNormalCone},  $v\in N_{S}(\bar{x};d)$.
\if{ According to \cite[Proposition 6.27(b)]{RW98}, there exists $v_n\in N_S(x+t_nd_n)\cap \mathbb{S}$ such that
 \begin{equation}\label{com-3-1}
 \langle w,v_n \rangle {=} {\rm dist}(w, T_{S}(x+t_nd_n))>\epsilon
 \end{equation}
 for all $n\in \mathcal{N}$. Since
 $v_n\in N_S(x+t_nd_n),$
 %=\limsup\limits_{z\overset{S}{\to} x+t_nd_n}\widehat{N}_S(z)
 by definition of the limiting normal cone,  for any fixed $n$  there exists $z_n^k \overset{S} \to x+t_nd_n$  and $v^k_n\to v_n$ with $v^k_n\in \widehat{N}_S(z_n^k)$ as $k\to \infty$.
 Thus there exists $k(n)$ such that $\|z_n^{k(n)}-(x+t_nd_n)\|\leq t_n/2^n$ and
 $\|v_n^{k(n)}-v_n\|\leq \epsilon/2(\|w\|+1)$. It follows that
 $\langle w, v_n^{k(n)} \rangle=\langle w, v_n \rangle+\langle w,v_n^{k(n)}-v_n \rangle \geq \epsilon -\|w\|\|v_n^{k(n)}-v_n\| >\frac{\epsilon}{2}$. The boundedness of the sequence $\{v_n\}$ ensures that of $\{v^{k(n)}_n\}$.
 Hence we can assume by further taking subsequence if necessary that $v^{k(n)}_n\to v$ as $n\to \infty$. Let $d'_n:=d_n+\frac{z^{k(n)}_n-(x+t_nd_n)}{t_n}$.
 Then $d'_n\to d$ since o $d_n\to d$ and $\frac{z^{k(n)}_n-(x+t_nd_n)}{t_n}\to 0$. Note that $z^{k(n)}_n=(x+t_nd_n)+(z^{k(n)}_n-(x+t_nd_n))=
 x+t_n(d_n+\frac{z^{k(n)}_n-(x+t_nd_n)}{t_n})=x+t_nd'_n$.
 Since $v^{k(n)}_n\in \widehat{N}_S(z^{k(n)}_n)=\widehat{N}_S(x+t_nd'_n)$ and $d'_n\to d$, we have that  $v\in N_S(x;d)$ by definition of directional normal cone.
 }\fi
Hence we obtain $\langle w,v \rangle\geq \epsilon$ by (\ref{com-3-1}). So $w\notin N_{S}(x;d)^\circ$.

 Thus we have shown $\widehat{T}_S(x;d)^{\rm comp}\subseteq (N_S(x;d)^\circ)^{\rm comp}$ and the inclusion $\widehat{T}_S(x;d)\supseteq N_S(x;d)^\circ$ follows. We conclude $\widehat{T}_S(x;d)=N_S(x;d)^\circ$ and therefore the directional regular tangent cone is closed and convex as a polar cone.

The rest of proofs follow from the fact that  the  directional Clarke normal cone  is closed and convex as well.
  \end{proof}

% \begin{lemma}\label{lem-cone-1}\cite[Lemma 2.1]{BenkoHelmutJiri}
% For a closed set $A$ and $x\in A$,
% \[
% N_A(x)=\widehat{N}_A(x)\cup \bigcup\limits_{d \in T_A(x)\cap \mathbb{S} }N_A(x;d)
% \]
% \end{lemma}
%\begin{corollary}\label{Cor2.1}
% For a closed cone $A$ and $x\in A$,
% \[
% N_A(0)=\widehat{N}_A(0)\cup\bigcup\limits_{d\in T_A(x)\cap \mathbb{S}}N_A(0;d)=\widehat{N}_A(0)
% \cup\bigcup\limits_{d\in T_A(x)\cap \mathbb{S}} N_A(d)
% \]
% \end{corollary}
% \begin{proof}
% The result follows from Lemma \ref{lem-cone-1} and the fact
% $N_A(0;d)=N_A(d)$ by \cite[Proposition 3.4]{YZ17}.
% \end{proof}
%\begin{corollary}\label{corollary-normal}
% For a closed set $A$ and $x\in A$, then
% \[
% N_{T_A(x)}(0)=\widehat{N}_A(x)\cup\bigcup\limits_{d\neq 0}N_{T_A(x)}(d).
% \]
% \end{corollary}
%
% \begin{proof}
% Since $T_A(x)$ is a closed cone,  by Corollary \ref{Cor2.1},
% \begin{eqnarray*}
% N_{T_A(x)}(0)&=&\widehat{N}_{T_A(x)}(0)\cup\bigcup\limits_{d\neq 0}N_{T_A(x)}(d) \\
% &=& T^\circ_A(x)\cup\bigcup\limits_{d\neq 0}N_{T_A(x)}(d)\\
% &=&  \widehat{N}_A(x)\cup\bigcup\limits_{d\neq 0}N_{T_A(x)}(d),
% \end{eqnarray*}
% where the first equality comes from Corollary \ref{corollary-normal} and the second equality follows the fact that $T_A(x)$ is a closed cone.
% \end{proof}

In this paper, we rely on the following stability property of a set-valued map in developing our results.
\begin{definition}\cite[Definition 1]{Gfr13a}
%\label{directionMS}
Let $\varphi :\Re^n\rightarrow  \Re^m$,  ${C} \subseteq \Re^m$ and $\varphi(\bar x) \in C$.
 We say that the  set-valued map
 $M(x):=\varphi(x)-C$ is metrically subregular (MS) at $(\bar{x},0)$ { in direction $d\in \mathbb{R}^n$}, if there exist
 $\kappa, \rho, \delta>0$ such that
 \begin{equation}\label{EqDirMS}
 {\rm dist}(x,M^{-1}(0)) \leq \kappa \, {\rm dist}(\varphi(x), C), \quad \forall x\in \bar{x}+V_{\rho,\delta}(d),
 \end{equation}
 where
 \begin{eqnarray*}
 V_{\rho,\delta}(d)&:=&\left\{w\in \rho \mathbb{B}\left|  \big\|\|d\|w-\|w\|d\big\|\leq \delta \|w\|\|d\| \right. \right\}\\
 &=& \left \{\begin{array}{ll}
 \rho \mathbb{B} & \mbox{ if } d=0\\
{\{0\}\cup}\{w\in  \rho \mathbb{B}\backslash \{0\} | \| \frac{w}{\|w  \|} -\frac{d}{\|d\|}\| \leq \delta \} & \mbox{ if } d\not =0
 \end{array} \right .
 \end{eqnarray*} is a directional neighborhood of the direction $d$.
 In the case where  $d=0$, we simple say that $M$ is metrically subregular at $(\bar{x},0)$ or metric subregularity constraint qualification (MSCQ) holds at $\bar x$.

{The infimum of  $\kappa$  over all such combinations of $\kappa$, $\rho$ and $\delta$ fulfilling \eqref{EqDirMS} is called the modulus of the respective property.}
  \end{definition}
  It is well-known that the metric subregularity of a set-valued map $M$ at $(\bar{x},0)$ is equivalent to the property of calmness/pseudo upper-Lipschitz continuity  of the inverse map $M^{-1}$ at $(0,\bar x)$; see  \cite{RW98,yy} for definition and \cite{DoRo14} for discussions about the equivalence.
\begin{proposition}\cite[Theorem 1]{HelmutKlatte} \label{FOSCMS} Let $\varphi :\Re^n\rightarrow  \Re^m$ be continuously differentiable,  $C\subseteq \Re^m$ be closed  and $\varphi(\bar x) \in C$. The  set-valued map
 $M(x):=\varphi(x)-C$ is metrically subregular  at $(\bar{x},0)$ in direction $d$ satisfying $\nabla \varphi(\bar x)d \in T_C(\varphi(\bar x))$ if the first order sufficient condition for metric subregularity (FOSCMS) for direction $d$ holds:
 \[ \nabla \varphi(\bar x)^T \lambda=0,\ \lambda \in N_C(\varphi(\bar x);\nabla \varphi(\bar x)d)\ \Longrightarrow\ \lambda=0.\]
\end{proposition}
Recently, weaker sufficient conditions than FOSCMS such as the directional quasi/pseudo-normality was introduced in \cite{BaiYe}. More sufficient conditions based on directional normal cones {or/and} for specific systems can be found e.g. in {\cite{HelmutKlatte,GY17,YZ17}}.

Classical sufficient conditions for metric subregularity include the case where $\varphi$ is affine and $C$ is the union of finitely many polyhedral sets by Robinson's polyhedral multifunction theory \cite{Robinson81} and the no nonzero abnormal multiplier constraint qualification (NNAMCQ) holds:
\begin{equation*}
\nabla \varphi(\bar x)^T \lambda=0, \ \lambda\in N_C(\varphi (\bar x))\Longrightarrow \lambda=0,\end{equation*}
by Mordukhovich criteria for metric regularity (see e.g., \cite[Theorem 9.40]{RW98}).
Note that when $C$ is convex,
 by \cite[Corollary 2.98]{BS},  NNAMCQ is equivalent to  Robinson's constraint qualification \cite[(2.178)]{BS}
$$0\in {\rm int}\{ \varphi(\bar x)+\nabla \varphi(\bar x)\mathbb{R}^n -C\},$$
{which in turn is equivalent to
\[\nabla \varphi(\bar x)\R^n+T_C(\varphi(\bar x))=\R^m\]
in finite dimensions.}

In the following result, we show that the directional metric subregularity of the mapping $M(x):=\varphi(x)-C$ is carried over to its linearized mapping, the so-called graphical derivative. Recall that for a set-valued mapping $M:\mathbb{R}^n\rightrightarrows\mathbb{R}^m$ the {\em graphical derivative} to $M$ at a point $(\bar x, \bar y)\in {\rm gph\,} M$ is the mapping $DM(\bar x,\bar y):\mathbb{R}^n\rightrightarrows\mathbb{R}^m$ satisfying
\[{\rm gph\,}DM(\bar x,\bar y)=T_{{\rm gph\,}M}(\bar x,\bar y),\]
resulting in $D(\varphi(\cdot)-C)(\bar x,\bar y)(w)=\nabla \varphi(\bar x)w-T_C(\varphi(\bar x)-\bar y).$
\begin{lemma}\label{LemSubregDM}
  Let $\varphi :\Re^n\rightarrow  \Re^m$ be continuously differentiable,  ${C} \subseteq \Re^m$ be closed  and $\varphi(\bar x) \in C$. If $M(x):=\varphi(x)-C$ is metrically subregular  at $(\bar x,0)$ in direction $$d\in\nabla \varphi(\bar x)^{-1}(T_C(\varphi(\bar x)):=\{d| \nabla \varphi(\bar x)d \in T_C(\varphi(\bar x))\}$$ with modulus $\bar\kappa$, then the graphical derivative $DM(\bar x,0)(w)=\nabla \varphi(\bar x)w-T_C(\varphi(\bar x))$ is metrically subregular at $(d,0)$ with modulus no larger than $\bar\kappa$.
\end{lemma}
\begin{proof}
  Choose $\kappa,\rho,\delta>0$ such that \eqref{EqDirMS} holds and consider $w\in V_{\rho,\delta}(d)$ together with $v\in T_C(\varphi(\bar x))$ satisfying $\dist\big(\nabla \varphi(\bar x)w,T_C(\varphi(\bar x))\big)=\norm{\nabla \varphi(\bar x)w-v}$. Then $tw\in V_{\rho,\delta}(d)$ for all $t\in[0,1]$ and, by picking a sequence $t_k\downarrow 0$ with $\dist(\varphi(\bar x)+t_k v,C)=o(t_k)$, we have
  \begin{eqnarray*} \dist(\bar x+t_k w,M^{-1}(0))
  & \leq & \kappa\dist(\varphi(\bar x+t_kw),C)\\
  & =& \kappa\big(\dist(\varphi(\bar x)+t_k\nabla \varphi(\bar x)w+o(t_k),C)\big)\\
  & =& \kappa\big(\dist(\varphi(\bar x)+t_k\nabla \varphi(\bar x)w,C)+o(t_k)\big)\\
  & \leq &\kappa\big(t_k\norm{\nabla \varphi(\bar x)w-v}+o(t_k)\big),\end{eqnarray*}
  where the second equality follows from the Lipschitz property of the distance function.

  Thus we can find a sequence $x_k$ satisfying $\varphi(x_k)\in C$ and
  \begin{equation}\label{EqAuxDirMS}\norm{x_k-(\bar x+t_k w)}\leq\kappa\big(t_k\norm{\nabla \varphi(\bar x)w-v}+o(t_k)\big).\end{equation}
  It follows that $(x_k-\bar x)/t_k$ is bounded and, by possibly passing to a subsequence, we may assume that $(x_k-\bar x)/t_k$ converges to some $w'$. Dividing \eqref{EqAuxDirMS} by $t_k$ and passing to the limit we obtain
  $\norm{w-w'}\leq\kappa\norm{\nabla \varphi(\bar x)w-v}$. Further, $$\dist(\varphi(\bar x)+t_k\nabla\varphi(\bar x)w',C)\leq \norm{\varphi(\bar x)+t_k\nabla\varphi(\bar x)w'-\varphi(x_k)}=o(t_k)$$ showing $\nabla\varphi(\bar x)w'\in T_C(\varphi(\bar x))$ and thus $w'\in DM(\bar x,0)^{-1}(0)$. Since the directional neighborhood $V_{\rho,\delta}(d)$ is
also a neighborhood of $d$ in the classical sense,  we can find some $\rho'>0$ such that $d+V_{\rho',\delta}(0)=d+\rho'\mathbb{B}\subseteq V_{\rho,\delta}(d)$. Thus we have shown that for all $w\in d+V_{\rho',\delta}(0)$ there holds
  \[\dist(w,DM(\bar x,0)^{-1}(0))\leq \|w-w'\| \leq \kappa \norm{\nabla \varphi(\bar x)w-v}=\kappa \dist\big(\nabla \varphi(\bar x)w,T_C(\varphi(\bar x))\big)\]
  showing metric subregularity of $DM(\bar x,0)$ at $(d,0)$
  \end{proof}
\subsection{Uniform MSCQ for the  second-order linearized  mapping}
From now on we denote by $\mathcal{F}$  the feasible region of problem (P), i.e., $\mathcal{F}:=\{x\,|\, g(x)\in \Lambda\}=M^{-1}(0)$, where $M(x):=g(x)-\Lambda$ is the feasible mapping.
If   the feasible mapping $M(x)$ is metrically subregular at $(\bar{x},0)$,  then
\[T_{\mathcal{F}}(\bar x)=DM(\bar x,0)^{-1}(0)=\{d\,|\, \nabla g(\bar x)d\in T_\Lambda(g(\bar x))\},\]
see, e.g., \cite[Proposition 1]{HO05} or \cite[Corollary {4.2}]{GO16-2}, where $DM(\bar x,0)(d)=\nabla g(\bar x)d-T_\Lambda(g(\bar x))$ denotes the  graphical derivative of $M$ at $(\bar x,0)$. Moreover,
by  \cite[Lemmas 3 and  4]{Helmut18} (a variant of \cite[Proposition 2.1]{Gfr11}), there is some $\kappa>0$ such that
\[\dist(d, T_{\mathcal{F}}(\bar x))\leq \kappa\dist(0, DM(\bar x,0)(d))=\kappa\dist(\nabla g(\bar x)d,T_\Lambda(g(\bar x)))\ \ \forall d\in\Re^n,\]
which is some kind of uniform metric subregularity of the graphical derivative. We will now show that an analogous relation holds for the second-order tangent set $T^2_{\mathcal{F}}(\bar x;d)$ and the second-order linearized mapping
%then by  \cite[Lemmas 3 and  4]{Helmut18} (a variant of \cite[Proposition 2.1]{Gfr11}),  its first-order linearized mapping $DM(\bar x,0)(w):=\nabla g(\bar x)w-T_\Lambda(g(\bar x))$ is  uniformly metric subregular  on $\Re^n$.
%If $M$ is metrically subregular at $(\bar x,0)$ then
%\[T_{\mathcal{F}}(\bar x)
%%=DM(\bar x,0)^{-1}(0)
%=\{d\,|\, \nabla g(\bar x)d\in T_\Lambda(g(\bar x))\}\]
%see, e.g., \cite[Proposition 1]{HO05} or \cite[Corollary 2.16]{GO16-2}.
%% where $DM(\bar x,0)(d)=\nabla g(\bar x)d-T_\Lambda(g(\bar x))$ denotes the so-called graphical derivative of $M$ at $(\bar x,0)$.
%Moreover,
%by  \cite[Lemmas 3 and  4]{Helmut18} (a variant of \cite[Proposition 2.1]{Gfr11}), there is some $\kappa>0$ such that
%\[\dist(d, T_{\mathcal{F}}(\bar x))\leq \kappa\dist(0, DM(\bar x,0)(d))=\kappa\dist(\nabla g(\bar x)d,T_\Lambda(g(\bar x)))\ \forall d\in\Re^n,\]
%which is some kind of uniform metric subregularity of the graphical derivative. We will now show that an analogous relation holds for the second-order tangent set $T^{\red{2}}_{\mathcal{F}}(\bar x;d)$ and
\begin{equation}\label{EqSecOrdLin}
  D^2M(\bar x,0;d)(w):=\nabla g(\bar x)w+\nabla^2g(\bar x)(d,d)-T^2_\Lambda\big(g(\bar x);\nabla g(\bar x)d \big).
\end{equation}

To prove the result, we need the following lemma.
 \begin{lemma}\label{lem-distance}
 For any $x\in \Re^n$ and a set-valued mapping $C:\Re^m\rightrightarrows \Re^n$, one has
 \[
 \liminf_{u'\to u} {\rm dist}(x,C(u'))={\rm dist} (x,\limsup_{u'\to u}C(u')).
 \]
 \end{lemma}

 \begin{proof}
 Let $\{u_n\}$ be a sequence  satisfying
 \[ \liminf_{n\to \infty} {\rm dist}(x,C(u_n))= \liminf_{u'\to u} {\rm dist}(x,C(u')).\]
  Then according to \cite[Exercise 4.8]{RW98} we have
 \[\liminf_{n\to \infty} {\rm dist}(x,C(u_n))={\rm dist} (x,\limsup_{n\to \infty}C(u_n))\geq {\rm dist} (x,\limsup_{u'\to u}C(u')).\]
  Hence
 \[\liminf_{u'\to u} {\rm dist}(x,C(u'))\geq {\rm dist} (x,\limsup_{u'\to u}C(u')).\]

 Conversely, take $r$ satisfying $r>{\rm dist} (x,\limsup_{u'\to u}C(u'))$. Then there exists $x'\in \limsup_{u'\to u}C(u')$ such that $\|x-x'\|<r$. Since $x'\in \limsup_{u'\to u}C(u')$, then there exists $u_n\to u$ and $x'_n\in C(u_n)$ and $x'_n\to x'$. So $\|x-x'_n\|<r$ as $n$ large enough. Hence
 ${\rm dist}(x,C(u_n))\leq \|x-x'_n\|<r$. So
 \[r\geq \liminf_{n\to \infty}{\rm dist}(x,C(u_n))\geq \liminf_{u'\to u}{\rm dist}(x,C(u')).\]
 Due to the arbitrariness of $r>{\rm dist} (x,\limsup_{u'\to u}C(u'))$, we obtain
 \[{\rm dist} (x,\limsup_{u'\to u}C(u'))\geq \liminf_{u'\to u}{\rm dist}(x,C(u')).\]
 \end{proof}
\if{ \begin{lemma}\label{Subproblem-MSCQ}
Suppose that  the set-valued map $M(x):=g(x)-\Lambda$ is metrically subregular at $(\bar{x},0)$ in direction $d$ with $\bar{x}\in \mathcal{F}$ and $\nabla g(\bar x)d\in T_{\Lambda}(g(\bar{x}))$. Then the set-valued map
 \[
 M_d:w\rightrightarrows \nabla g(\bar{x})w+\nabla g^2(\bar{x})(d,d)- T^2_{\Lambda}\big(g(\bar{x}); \nabla g(\bar{x})d\big)
 \]
is uniformly metric subregular on $\mathbb{R}^n$, i.e., there exists  $\kappa>0, \rho>0, \delta>0$ such that
 \begin{equation*}\label{ms-2-tangent}
 {\rm dist}(w',T_{\cal F}^2(x;d))\leq \tau {\rm dist}\left(\nabla g(\bar{x})w'+\nabla g^2(\bar{x})(d,d), T^2_{\Lambda}\Big(g(\bar{x});\nabla g(\bar{x})d\Big) \right), \ \ \forall w'\in  \mathbb{R}^n.
 \end{equation*}
 \end{lemma}

 \begin{lemma}\label{lem-distance}
 For any $x\in \Re^n$ and a set-valued mapping $C:\Re^m\rightrightarrows \Re^n$, one has
 \[
 \liminf_{u'\to u} {\rm dist}(x,C(u'))={\rm dist} (x,\limsup_{u'\to u}C(u')).
 \]
 \end{lemma}

 \begin{proof}
 Let $\{u_n\}$ be a sequence  satisfying
 \[ \liminf_{n\to \infty} {\rm dist}(x,C(u_n))= \liminf_{u'\to u} {\rm dist}(x,C(u')).\]
  Then according to \cite[Exercise 4.8]{RW98} we have
 \[\liminf_{n\to \infty} {\rm dist}(x,C(u_n))={\rm dist} (x,\limsup_{n\to \infty}C(u_n))\geq {\rm dist} (x,\limsup_{u'\to u}C(u')).\]
  Hence
 \[\liminf_{u'\to u} {\rm dist}(x,C(u'))\geq {\rm dist} (x,\limsup_{u'\to u}C(u')).\]

 Conversely, take $r$ satisfying $r>{\rm dist} (x,\limsup_{u'\to u}C(u'))$. Then there exists $x'\in \limsup_{u'\to u}C(u')$ such that $\|x-x'\|<r$. Since $x'\in \limsup_{u'\to u}C(u')$, then there exists $u_n\to u$ and $x'_n\in C(u_n)$ and $x'_n\to x'$. So $\|x-x'_n\|<r$ as $n$ large enough. Hence
 ${\rm dist}(x,C(u_n))\leq \|x-x'_n\|<r$. So
 \[r\geq \liminf_{n\to \infty}{\rm dist}(x,C(u_n))\geq \liminf_{u'\to u}{\rm dist}(x,C(u')).\]
 Due to the arbitrariness of $r>{\rm dist} (x,\limsup_{u'\to u}C(u'))$, we obtain
 \[{\rm dist} (x,\limsup_{u'\to u}C(u'))\geq \liminf_{u'\to u}{\rm dist}(x,C(u')).\]
 \end{proof}
  \begin{proof}
Since $g(x)-\Lambda$ is metrically subregular at
 $(\bar{x},0)$ in direction $d$, by Definition \ref{directionMS},  there exist
 $\kappa, \rho, \delta>0$ such that
 \[
 {\rm dist}(x,{\cal F}) \leq \kappa \, {\rm dist}(g( x) , \Lambda), \quad \forall x-\bar{x}\in V_{\rho,\delta}(d).
 \]Let  $w'\in A$ be fixed and $x=:\bar x+td+\frac{1}{2}t^2w'$. Then for $t>0$ sufficiently small,
 $x-\bar x=td+\frac{1}{2}t^2w' \in V_{\rho,\delta}(d)$. It follows that
 \begin{eqnarray*}
\lefteqn{{\rm dist}(\bar{x}+td+\frac{1}{2}t^2w',{\cal F})}\\
  &\leq & \kappa {\rm dist}(g(\bar{x}+td+\frac{1}{2}t^2w'),\Lambda)\\
 &=& \kappa{\rm dist}\Big(g(\bar{x})+t\nabla g(\bar{x})d+\frac{1}{2}t^2\big(\nabla g(\bar{x})w'+\nabla^2g(\bar{x})(d,d)\big)+o(t^2),\Lambda\Big)\\
 &=& \kappa{\rm dist}\Big(g(\bar{x})+t\nabla g(\bar{x})d+\frac{1}{2}t^2\big(\nabla g(\bar{x})w'+\nabla^2g(\bar{x})(d,d)\big),\Lambda\Big)+o(t^2),
 \end{eqnarray*}
 where the first and second equalities follow  from Taylor expansion and the Lipschitz continuity of the distance function. Dividing both sides of the above inequality by $\frac{1}{2} t^2$ we obtain
 \[
 {\rm dist}\left(w',\frac{{\cal F}-\bar{x}-t d}{\frac{1}{2}t^2}\right)\leq \kappa {\rm dist} \left(\nabla g(\bar{x})w'+\nabla^2 g(\bar{x})(d,d),\frac{\Lambda-g(\bar{x})-t \nabla g(\bar{x})d}{\frac{1}{2}t^2}\right)+\frac{o(t^2)}{\frac{1}{2}t^2}.
 \]
 Taking the inf-limits on the both sides of the above inequality and using Lemma \ref{lem-distance} yields
 \begin{eqnarray*}
\lefteqn{{\rm dist}(w',T_{\cal F}^2(\bar x;d))}\\
 &=&{\rm dist}\left(w',\limsup_{t\downarrow 0}\frac{{\cal F}-\bar{x}-t d}{\frac{1}{2}t^2}\right)=
 \liminf_{t\downarrow 0}{\rm dist}\left(w',\frac{{\cal F}-\bar{x}-t d}{\frac{1}{2}t^2}\right)\\
 & \leq & \kappa\, \liminf_{t\downarrow 0}{\rm dist} \left(\nabla g(\bar{x})w'+\nabla^2g(\bar{x})(d,d),\frac{\Lambda-g(\bar{x})-t \nabla g(\bar{x})d}{\frac{1}{2}t^2}\right)\\
 &=& \kappa\, {\rm dist}\left(\nabla g(\bar{x})w'+\nabla^2g(\bar{x})(d,d),\limsup_{t\downarrow 0}\frac{\Lambda-g(\bar{x})-t \nabla g(\bar{x})d}{\frac{1}{2}t^2}\right)\\
 &=& \kappa\, {\rm dist}\Big(\nabla g(\bar{x})w'+\nabla^2g(\bar{x})(d,d),T_{\Lambda}^2(g(\bar{x});\nabla g(\bar{x})d)\Big).
 \end{eqnarray*}
 \end{proof}
}\fi

\begin{proposition}\label{PropSecOrderTanSet}
  Let  $\bar{x}\in \mathcal{F}$ and suppose that  the set-valued map $M(x):=g(x)-\Lambda$ is metrically subregular at $(\bar{x},0)$ in direction $d$ with modulus $\kappa$. Then
  \begin{equation}\label{EqTanF}d\in T_{\mathcal F}(\bar x)\ \Leftrightarrow\ \nabla g(\bar x)d\in T_\Lambda(g(\bar x))\end{equation}
  and for any $d$ satisfying $g(\bar x)d\in T_\Lambda(g(\bar x))$, one has
  \begin{equation}\label{EqSecOrdTanSetF}T^2_{\mathcal{F}}(\bar x;d)=\{w\,|\, \nabla g(\bar x)w+\nabla^2g(\bar x)(d,d)\in T^2_\Lambda\big(g(\bar x);\nabla g(\bar x)d\big)\}=D^2M(\bar x,0;d)^{-1}(0)\end{equation}
  with $D^2M(\bar x,0;d)$ given by \eqref{EqSecOrdLin}. Moreover,
  \begin{align}
    \label{EqMetrSubrSecOrdLinMap}\dist(w ,T_{\cal F}^2(x;d))&\leq \kappa \dist\Big(\nabla g(\bar{x})w+\nabla g^2(\bar{x})(d,d), T^2_{\Lambda}\big(g(\bar{x});\nabla g(\bar{x})d\big) \Big)\\
    \nonumber &=\kappa\dist\big(0,D^2M(\bar x,0;d)(w)\big)\ \forall w\in  \mathbb{R}^n. \end{align}
\end{proposition}
\begin{proof}
The equivalence \eqref{EqTanF} follows from \cite[Proposition {4.1}]{GO16-2}. Next we will prove the inequality \eqref{EqMetrSubrSecOrdLinMap}. Let  $w\in\Re^n$ be fixed and consider $\kappa'>\kappa$. It follows that for $t>0$ sufficiently small
 \begin{eqnarray*}
\lefteqn{{\rm dist}(\bar{x}+td+\frac{1}{2}t^2w,{\cal F})}\\
  &\leq & \kappa' {\rm dist}(g(\bar{x}+td+\frac{1}{2}t^2w),\Lambda)\\
 &=& \kappa'{\rm dist}\Big(g(\bar{x})+t\nabla g(\bar{x})d+\frac{1}{2}t^2\big(\nabla g(\bar{x})w+\nabla^2g(\bar{x})(d,d)\big)+o(t^2),\Lambda\Big)\\
 &=& \kappa'{\rm dist}\Big(g(\bar{x})+t\nabla g(\bar{x})d+\frac{1}{2}t^2\big(\nabla g(\bar{x})w+\nabla^2g(\bar{x})(d,d)\big),\Lambda\Big)+o(t^2),
 \end{eqnarray*}
 where the first and second equalities follow  from Taylor expansion and the Lipschitz continuity of the distance function. Dividing both sides of the above inequality by $\frac{1}{2} t^2$ we obtain
 \[
 {\rm dist}\left(w,\frac{{\cal F}-\bar{x}-t d}{\frac{1}{2}t^2}\right)\leq \kappa' {\rm dist} \left(\nabla g(\bar{x})w+\nabla^2 g(\bar{x})(d,d),\frac{\Lambda-g(\bar{x})-t \nabla g(\bar{x})d}{\frac{1}{2}t^2}\right)+\frac{o(t^2)}{\frac{1}{2}t^2}.
 \]
 Taking the inf-limits on the both sides of the above inequality and using Lemma \ref{lem-distance} yields
 \begin{eqnarray*}
\lefteqn{{\rm dist}(w,T_{\cal F}^2(\bar x;d))}\\
 &=&{\rm dist}\left(w,\limsup_{t\downarrow 0}\frac{{\cal F}-\bar{x}-t d}{\frac{1}{2}t^2}\right)=
 \liminf_{t\downarrow 0}{\rm dist}\left(w,\frac{{\cal F}-\bar{x}-t d}{\frac{1}{2}t^2}\right)\\
 & \leq & \kappa'\, \liminf_{t\downarrow 0}{\rm dist} \left(\nabla g(\bar{x})w+\nabla^2g(\bar{x})(d,d),\frac{\Lambda-g(\bar{x})-t \nabla g(\bar{x})d}{\frac{1}{2}t^2}\right)\\
 &=& \kappa'\, {\rm dist}\left(\nabla g(\bar{x})w+\nabla^2g(\bar{x})(d,d),\limsup_{t\downarrow 0}\frac{\Lambda-g(\bar{x})-t \nabla g(\bar{x})d}{\frac{1}{2}t^2}\right)\\
 &=& \kappa'\, {\rm dist}\Big(\nabla g(\bar{x})w+\nabla^2g(\bar{x})(d,d),T_{\Lambda}^2(g(\bar{x});\nabla g(\bar{x})d)\Big).
 \end{eqnarray*}
 Since $\kappa'>\kappa$ can be chosen arbitrarily close to $\kappa$, the bound \eqref{EqMetrSubrSecOrdLinMap} follows. From \eqref{EqMetrSubrSecOrdLinMap} we may conclude
 \[T^2_{\mathcal{F}}(\bar x;d)\supseteq \{w\,|\, \nabla g(\bar x)w+\nabla^2g(\bar x)(d,d)\in T^2_\Lambda\big(g(\bar x);\nabla g(\bar x)d\big)\}.\]
 There remains to show the reverse inclusion. Consider $w\in T^2_{\mathcal{F}}(\bar x;d)$ together with sequences $t_k\downarrow 0$ and $w_k\to w$ with $\bar x+t_k d+\frac 12t_k^2w_k\in \mathcal{F}$. Then
 \[g(\bar x+t_kd+\frac 12t_k^2w_k)=g(\bar x)+ t_k\nabla g(\bar x)d+\frac 12t_k^2\big(\nabla g(\bar x)w+\nabla^2g(\bar x)(d,d)\big)+o(t_k^2)\in\Lambda\]
 implying $\nabla g(\bar x)w+\nabla^2g(\bar x)(d,d)\in T^2_\Lambda\big(g(\bar x);\nabla g(\bar x)d\big)$. This shows the inclusion $\subseteq$ in \eqref{EqSecOrdTanSetF} and
  the proof of the proposition is complete.
\end{proof}
The chain rule \eqref{EqSecOrdTanSetF} was derived in \cite[Proposition 13.13]{RW98} under the assumption of metric regularity of $M$ and in \cite{Meh19} under (non-directional) metric subregularity, see also \cite{MoMoSa19}. For a related result under directional metric subregularity we refer to \cite[Proposition 4.1]{Pen98}.
\if{
Proposition \ref{PropSecOrderTanSet} tells us that the chain rule for second-order tangents derived in \cite[Proposition 13.13]{RW98} under the assumption of metric regularity of $M$ remains valid under the weaker assumption of directional metric subregularity.}
\fi

 %\green{The following  paragraph and the corollary are moved from section 3.}

 It is known that the second-order tangent set may be empty even if the set considered is a convex set \cite{BS}. As a consequence of Proposition \ref{PropSecOrderTanSet}, we can show  that $T_{\cal F}^2$ and $T^2_{\Lambda}$ are either empty or nonempty at the same time under the metric subregularity.
 \begin{corollary}\label{CorNonEmpty}
  Suppose that  the set-valued map $M(x):=g(x)-\Lambda$ is metrically subregular at $(\bar{x},0)$ in direction $d$ with $\bar{x}\in \mathcal{F}$ and $\nabla g(\bar x)d\in T_{\Lambda}(g(\bar{x}))$. Then $T_{\cal F}^2(x;d)\neq \emptyset$ if and only if $T^2_{\Lambda}\big(g(\bar{x});\nabla g(\bar{x})d\big)\neq \emptyset$.
 \end{corollary}
 \begin{proof}
 Suppose  first that   $T^2_{\Lambda}\big(g(\bar{x});\nabla g(\bar{x})d\big)\neq \emptyset$. Then $T_{\cal F}^2(x;d)\neq \emptyset$
 by virtue of  \eqref{EqMetrSubrSecOrdLinMap}, because otherwise the left-hand side of \eqref{EqMetrSubrSecOrdLinMap} must be equal to infinity while the right-hand side is finite which is impossible.
 The reverse statement  follows  immediately from \eqref{EqSecOrdTanSetF}.
 \end{proof}

 \subsection{On  estimates of the  normal cone of  the   tangent  sets}

In this subsection we give some estimates on the limiting normal cone to the first and the second-order tangent set  which will be used in the necessary optimality condition we are developing.

 \begin{lemma}\label{relationship-normal-cone} Let $S$ be a closed subset  in $\Re^n$, $x\in S$, $d\in T_S(x)$ and $w\in T^2_{S}(x,d)$. Then
   \begin{eqnarray}\label{eq:1new}
 && N_{T_S(x)}(d)\subseteq N_{S}(x;d), \\
&&  N_{{T^2_{S}(x;d)}}(w)\subseteq N_{S}(x;d). \label{eq:2new}
 \end{eqnarray}
 \end{lemma}

 \begin{proof}
  We only prove (\ref{eq:2new}) since  (\ref{eq:1new}) can be proved similarly.
  Take $v\in N_{{T^2_{S}(x,d)}}(w)$. Note that $T_S^2(x;d)=\limsup\limits_{t\downarrow 0}\frac{S-x-td}{\frac{1}{2}t^2}$. It follows from \cite[Exercise 6.18]{RW98} that there exists $t_k\downarrow 0$, $w_k\in T_k:=\frac{S-x-t_kd}{\frac{1}{2}t_k^2}$ and $v_k\in \widehat N_{T_k}(w_k)$ with $v_k\to v$ and $w_k\to w$. Since
 $T_k=\{w|x+t_kd+\frac{1}{2}t_k^2w\in S\}$, by the change of coordinates formula in \cite[Exercise 6.7]{RW98} we have
$v_k\in \widehat{N}_{T_k}(w_k)=\widehat{N}_S\big(x+t_kd+\frac{1}{2}t_k^2w_k\big)=
 \widehat{N}_S\big(x+t_k(d+\frac{1}{2}t_kw_k)\big).$
 Hence $v\in N_{S}(x;d)$.
 \end{proof}

 The inclusion  (\ref{eq:1new}) can be strict. For example, take $S:=\{0,1,\frac{1}{2},\frac{1}{n},\dots\}$. It is easy to see that  $T_{S}(0)=\Re_+$ and $N_{T_{S}(0)}(0)=\Re_-$. Take $d=1\in T_{S}(0)$. Then
 $N_{S}(0;d)\supseteq \limsup_{n\to \infty}\widehat{N}_{S}(\frac{1}{n})=\Re$,  since $\frac{1}{n}$ is an isolated point in $S$. So $N_{S}(0;d)=\Re.$ Hence
 $\{0\}=N_{T_{S}(x)}(d)\varsubsetneq N_{S}(x;d)=\Re$ at $x=0$.
%\begin{remark}
 %Note

 %Note  that (\ref{eq:1new}) holds as an equality whenever $S$ is convex, as shown in the corollary below.
{By \eqref{EqDirLimNormalConeConvex}, the inclusion (\ref{eq:1new}) holds as an equality whenever $S$ is convex.}
%  In this case
% \[N_{T_S(x)}(d)=(T_S(x))^\circ\cap \{d\}^\perp=N_S(x)\cap \{d\}^\perp=N_S(x;d).\]
  However (\ref{eq:2new}) can fail to be an equality even if $S$ is  polyhedral; see e.g. Example \ref{ex2.2}.
   \begin{example}\label{ex2.2}
Let  $S=\Re^2_+$ and  $x=(0,0)$. Then $T_S(x)=S$. Take $d=(1,0)\in T_S(x)$. Then $T_S^2(x;d)=T_{T_S(x)}(d)=T_S(d)=\Re\times \Re_+$. Take $w=(0,1)\in T^2_S(x;d)$. Then
 $
 N_{T^2_S(x;d)}(w)=%0
{\{(0,0)\}}
 $
 and
 $N_{T_S(x)}(d)=N_S(d)=\{0\}\times \Re_-.$
 So
 $  N_{T^2_S(x;d)}(w)\varsubsetneq  N_{T_S(x)}(d).$
 \end{example}

 \section{Primal form of second-order necessary optimality conditions}

%------------------------------------------------------------------------------- Definition 3.1
In this section we derive the primal form of second-order necessary optimality conditions for the general problem (P) under directional metric subregularity.
%==============================================
\if{Such conditions are derived in \cite[Lemma 4.4]{BS} under Robinson's constraint qualification for the case where $\Lambda$ is  convex.
% of a Banach space.

 Define the critical cone at $\bar x$ as:
 \[
 C(\bar{x}):= \Big\{ d \, | \, \nabla g(\bar x)d \in T_{\Lambda}(g(\bar x)),
 \nabla f(\bar x) d \leq 0 \Big\}.
 \]
%  Under the metrical subregularity direction $d$,  condition we can derive an equivalent expression for the critical cone.
 \begin{lemma} \label{Lem3.1}Let  $\bar{x}$ be a locally optimal solution of (P). Suppose that the feasible set mapping $g(x)-\Lambda$ is metrically subregular at $(\bar x, 0)$ in direction $d\in C(\bar x)$. Then $\nabla f(\bar x) d =0$.
% \begin{equation} \label{equi-criticalcone} C(\bar{x})= \Big\{ d \, | \, \nabla g(\bar x)d \in T_{\Lambda}(g(\bar x)),
% \nabla f(\bar x) d =0 \Big\}.
% \end{equation}
 \end{lemma}
  \begin{proof} Since $\bar x$ is a locally optimal solution of (P), we have
 $$\nabla f(\bar x)d\geq 0 \qquad \forall d \in T_{\cal F}(\bar x).$$
By Proposition \ref{PropSecOrderTanSet} we have
 $d \in T_{\cal F}(\bar x) \Longleftrightarrow \nabla g(\bar x) d\in T_\Lambda(g(\bar x))$ and hence $\nabla f(\bar x)d\geq 0$. By definition of the critical cone, it follows that $\nabla f(\bar x)d= 0$.
 \end{proof}
 The following second-order necessary optimality condition in primal form improves \cite[Lemma 3.44]{BS} in that $\Lambda$ does not need to be convex and the result holds under the {directional} metric subregularity instead of Robinson's constraint qualification.
%---------------------------------------------------------------------------------- Theorem 3.1
 \begin{theorem} \label{Thm3.1}
Let  $\bar{x}$ be a locally optimal solution of (P). Suppose that  the feasible set mapping $g(x)-\Lambda$ is metrically subregular at $(\bar x, 0)$ in direction  $d\in C(\bar{x})$.
 Then
 \[
 \nabla f(\bar x)w+\nabla^2f(\bar x)(d,d) \geq 0, \quad \forall w\in T^2_{\cal F}(\bar{x};d).
 \]
 \end{theorem}
 \begin{proof}
 There is nothing to prove if $T^2_{\cal F}(\bar{x};d)$ is empty. Hence let us consider the case that
 $T^2_{\cal F}(\bar{x};d)\neq \emptyset$.
Let $w\in T^2_{\cal F}(\bar x;d)$. Then by definition there exists $t_n\downarrow 0$ and $w_n\to w$
 satisfying $\bar x+t_nd+\frac{1}{2}t_n^2w_n\in {\cal F}$. The local optimality of $\bar x$
 ensures that
 $
 f(\bar x+t_nd+\frac{1}{2}t_n^2w_n) \geq f(\bar x)
 $
 whenever $n$ is large enough. Hence by using the second order Taylor expansion of $f$ at $\bar x$ we have
 \begin{eqnarray}
 f(\bar x)
 & \leq & f(\bar x+t_nd+\frac{1}{2}t_n^2w_n) \nonumber \\
 &=& f(\bar x)+t_n \nabla f(\bar x)d + \frac{1}{2}t^2_n
     \Big( \nabla f(\bar x) {w_n}+\nabla^2f(\bar{x})(d,d)\Big)+o(t^2_n) \nonumber \\
 &=& f(\bar{x})+\frac{1}{2}t^2_n\Big({\nabla} f(\bar{x}) {w_n} + {\nabla}^2f(\bar{x})(d,d)\Big)
     + o(t^2_n), \label{eq-1}
 \end{eqnarray}
 where in the second equality we have used the fact that $\nabla f(\bar{x})d=0$ by Lemma \ref{Lem3.1}.
 %$\nabla f(\bar{x})d\geq 0$ as $d\in T_{\cal F}(\bar{x})$ and $\nabla f(\bar{x})d\leq 0$
 Dividing by $t_n^2$ on both sides of
 (\ref{eq-1}) and taking the limit yield the desired result.
 \end{proof}
}\fi
%===========================================================
{Recall the following basic second-order necessary condition.
\begin{theorem}[{cf. \cite[Corollary 1.3]{Pen94}}]\label{ThBasicSecOrd}Let  $\bar{x}$ be a locally optimal solution of (P).
 Then for all $d\in T_{\cal F}(\bar x)$ with $\nabla f(\bar x)d=0$ one has
 \[
 \nabla f(\bar x)w+\nabla^2f(\bar x)(d,d) \geq 0, \quad \forall w\in T^2_{\cal F}(\bar{x};d).
 \]
 \end{theorem}
 We shall now apply this theorem under the assumption of directional metric subregularity.
 Define the critical cone at $\bar x$ as:
 \[
 C(\bar{x}):= \Big\{ d \, | \, \nabla g(\bar x)d \in T_{\Lambda}(g(\bar x)),
 \nabla f(\bar x) d \leq 0 \Big\}.
 \]
 \begin{lemma} \label{Lem3.1}Let  $\bar{x}$ be a locally optimal solution of (P). Suppose that the feasible set mapping $g(x)-\Lambda$ is metrically subregular at $(\bar x, 0)$ in direction $d\in C(\bar x)$. Then $\nabla f(\bar x) d =0$.
 \end{lemma}
  \begin{proof} Since $\bar x$ is a locally optimal solution of (P), we have
 $$\nabla f(\bar x)d\geq 0 \qquad \forall d \in T_{\cal F}(\bar x).$$
By Proposition \ref{PropSecOrderTanSet} we have
 $d \in T_{\cal F}(\bar x) \Longleftrightarrow \nabla g(\bar x) d\in T_\Lambda(g(\bar x))$ and hence $\nabla f(\bar x)d\geq 0$. By definition of the critical cone, it follows that $\nabla f(\bar x)d= 0$.
 \end{proof}
 The following second-order necessary optimality condition in primal form  follows now immediately from Lemma \ref{Lem3.1}, Theorem \ref{ThBasicSecOrd} and Proposition \ref{PropSecOrderTanSet}. It improves \cite[Lemma 3.44]{BS} in that $\Lambda$ does not need to be convex and the result holds under the {directional} metric subregularity instead of Robinson's constraint qualification.
%---------------------------------------------------------------------------------- Theorem 3.1
 \begin{corollary} \label{Cor3.1}
Let  $\bar{x}$ be a locally optimal solution of (P). Suppose that  the feasible set mapping $g(x)-\Lambda$ is metrically subregular at $(\bar x, 0)$ in direction  $d\in C(\bar{x})$.
 Then for all $w$ satisfying $\nabla g(\bar x)w +\nabla^2g(\bar x)(d,d)\in T^2_\Lambda(g(\bar x);\nabla g(\bar x)d)$ one has
 \[
 \nabla f(\bar x)w+\nabla^2f(\bar x)(d,d) \geq 0.
 \]
 \end{corollary}
}

 \section{Dual form of second-order optimality conditions}
%\label{SecSOC}

 In this section we will derive the dual form of second-order necessary optimality conditions for the general problem (P). {By Corollary \ref{Cor3.1} we know, that at a local solution $\bar x$ of (P) for every critical direction $d$ satisfying a directional metric subregularity constraint qualification, the optimal value of  the program
   \begin{eqnarray*}
 \min_w && \nabla f(\bar{x})w+\nabla^2f(\bar{x})(d,d) \\
 \mbox{s.t.} && \nabla g(\bar{x})w+\nabla^2g(\bar{x})(d,d)\in T^2_\Lambda(g(\bar x);\nabla g(\bar x)d)
 \end{eqnarray*}
 is nonnegative. The second-order necessary conditions for (P) presented below are necessary conditions and {characterizations}, respectively, of this fact.
 }

 % under metric subregularity condition.
 %When the set  $\Lambda$ is  convex, second-order optimality conditions involving the support function of second-order tangent cone $T^2_\Lambda$ are studied in \cite[Theorems 3.45, Proposition 3.46]{BS} under Robinson's constraint qualification.

Recall that for a set $S$, its support function is defined as $\sigma_S(\lambda):=\sup_{u\in S}\lambda^Tu$.
%\textcolor{red}{By definition, $ \sigma_S(0)=0$.  When $S$ is a closed convex set and $\lambda\not =0$, either the supremum $\sigma_S(\lambda)$ is achieved or $\sigma_S(\lambda)=\infty$.
Suppose that the supremum $\sigma_S(\lambda)$ is achieved at $\bar u\in S$. Since
$\bar u\in S$ is an optimal solution for $\sup_{u\in S}\lambda^Tu $ if and only if
$\lambda \in N_S( \bar u)$ as $S$ is convex,
% and one has
%$$\sigma_S(\lambda)=\lambda^T u \qquad \forall \lambda \in N_S( u).$$}
the  support function in convex case can be represented as
\begin{eqnarray}
\sigma_S(\lambda)= \lambda^Tu  \ \mbox{ if } \lambda \in N_S(u) \mbox{ for some } u\in S. \label{supportfunction}
%\left \{\begin{array}{ll}
%\lambda^Tu & \mbox{ if } \lambda \in N_S(u) \mbox{ for some } u\in S\\
%\infty & \mbox{otherwise }
%\end{array}\right . . \nonumber \\
%&=&\inf_u\{ \lambda^Tu \mid \lambda \in N_S(u)\}. \label{supportfunction}
\end{eqnarray}
% with the convention that the infimum taken over the empty set is $\infty$.
%It is obvious that when $S$ is non-convex, the above expression for the support function is no longer valid in general.
%In order to extend the classical second-order condition to allow non-convex set $S$, i
Inspired by the above expression for the support function when the set is convex and the supremum is achieved, we define the following  function which will play an important role for our analysis. It turns out that this function is in general smaller and coincides with the support function when the set is convex.
 \begin{definition}
   Given a nonempty closed set $S\subseteq \Re^n$ we define the {\em lower generalized support function to $S$} as the mapping $\hat\sigma_S:\Re^n\to\Re\cup\{\infty\}$ by
   \[
     \hat\sigma_S(\lambda):=\liminf_{\tilde\lambda \to\lambda}\inf_u\{\tilde  \lambda^Tu \mid \tilde\lambda\in N_S(u)\}.
   \]
   {Further, for every subset $A\subseteq \Re^n$ we define the {\em lower generalized support function to $S$ with respect to $A$} as the mapping $\hat\sigma_{S,A}:\Re^n\to\Re\cup\{\infty\}$ by
   \[
     \hat\sigma_{S,A}(\lambda):=\liminf_{\tilde\lambda \to\lambda}\inf_u\{\tilde  \lambda^Tu \mid u\in N_S^{-1}(\tilde\lambda)\cap A\}.
   \]}
 \end{definition}
By convention,  $\hat\sigma_S(\lambda):=+\infty$ if $S=\emptyset$.
By the definition, we have
$\hat \sigma_S\leq \hat\sigma_{S,A}$
 for every subset $A\subseteq\Re^n$
 and $\hat\sigma_{S,B}\leq \hat\sigma_{S,A}$ whenever $A\subseteq B\subseteq\Re^n$.
% \blue{In addition, for all $A\subseteq S\subseteq B$ with $A\not =\emptyset$,
% we have $\hat\sigma_{S,A} \geq \hat\sigma_{A}$ and $\hat\sigma_{S,B}=\hat\sigma_S$.
% Indeed,
% $$A\subseteq S \Rightarrow N_S(u) \subseteq N_A(u) \quad \mbox{ if } u\in A,$$ and so for given $\tilde  \lambda$
% $$\{\tilde  \lambda^Tu \mid \tilde\lambda \in N_S(u), u\in  A\} \subseteq \{\tilde  \lambda^Tu \mid \tilde\lambda \in N_A(u), u\in  A\}.$$
% Hence we have
% \begin{eqnarray*}
% \hat\sigma_{S,A}(\lambda)& :=& \liminf_{\tilde\lambda \to\lambda}\inf_u\{\tilde  \lambda^Tu \mid u\in N_S^{-1}(\tilde\lambda)\cap A\}\\
% &=& \liminf_{\tilde\lambda \to\lambda}\inf_u\{\tilde  \lambda^Tu \mid \tilde\lambda \in N_S(u), u\in  A\}\\
% &\geq & \liminf_{\tilde\lambda \to\lambda}\inf_u\{\tilde  \lambda^Tu \mid \tilde\lambda \in N_A(u), u\in  A\}\\
% &=& \hat\sigma_{A}(\lambda).
% \end{eqnarray*}}

%In fact we
{We} can show that the limiting normal cone in the above definition {of $\hat\sigma_S$} can be replaced by the regular normal cone.
  \begin{lemma}\label{LemHatSigma}Let $S\subseteq\Re^n$ be closed. Then
   \[
     \hat\sigma_S(\lambda)=\liminf_{\tilde\lambda\to\lambda}\inf_u\{ \tilde\lambda^Tu \mid \tilde\lambda\in \widehat N_S(u)\}, \ \ \ \forall \lambda\in\Re^n.
   \]
 \end{lemma}
 \begin{proof}
   It follows easily  that by the definition of the limiting normal cone that  we have $\widehat N_S(u)\subseteq N_S(u)$ and for every $u\in S$,  every $\tilde\lambda\in N_S(u)$ and every $\epsilon>0$ we can find $u_\epsilon$ and $\tilde\lambda_\epsilon\in\widehat N_S(u_\epsilon)$ such that
   \[\Vert u-u_\epsilon\Vert\leq\epsilon,\ \Vert\tilde\lambda-\tilde\lambda_\epsilon\Vert\leq\epsilon,\ \vert \tilde\lambda^Tu-\tilde\lambda_\epsilon^Tu_\epsilon\vert\leq\epsilon.\]
 \end{proof}
%In the following proposition we show that when the set is convex, the lower generalized support function coincides with the support function.
In the following proposition we show that the lower generalized support function is always less than or equal to the support function and that both functions coincide when the underlying set is convex.
 \begin{proposition}\label{PropHatSigma}
 \begin{enumerate}
 \item For every nonempty closed set $S\subseteq \Re^n$ one has
   \[\hat\sigma_S(\lambda)\leq \sigma_S(\lambda), \ \forall\lambda.\]
 \item
   For every nonempty closed convex set $S\subseteq \Re^n$ one has
   \[\hat\sigma_S(\lambda)=\sigma_S(\lambda), \ \forall\lambda.\]

%   \green{ \item For every nonempty closed cone $K\subseteq \Re^n$ (not necessarily convex) and every $z\in\Re^n$ one has
%    \[\hat\sigma_{z+K}(\lambda)=\begin{cases}\lambda^Tz&\mbox{if $\lambda\in N_K(0)$,}\\
%    \infty&\mbox{otherwise.}\end{cases}\]}

    \end{enumerate}
 \end{proposition}
 \begin{proof}
 \begin{enumerate}
   \item %{At $\lambda=0$, since $0\in N_S(u)$ provided $u\in S$  it is clear that
%   \[
%     \hat\sigma_S(0)=\liminf_{\tilde\lambda\to 0}\inf_u\{ \tilde\lambda^Tu \mid \tilde\lambda\in N_S(u)\}\leq 0=\sigma_S(0).
%   \]
   % Now consider $\lambda \not =0$.
   If $\sigma_S(\lambda)=\infty$, then there is nothing to prove. Now assume that $\sigma_S(\lambda)<\infty$.
Then for any $\epsilon>0$,
   there exists an $\epsilon$-optimal solution, say $u_\epsilon$ satisfying
   $$-\lambda^T u_\epsilon +\delta_S(u_\epsilon)< -\sigma_S(\lambda)+\epsilon,  $$
   where $\delta_S(\cdot)$ denotes the indicator function of set $S$.
   By Ekeland's variational principle, for any $\mu>0$ satisfying $\mu(\|u_\epsilon\|+1)\leq \epsilon$,
   there exists
    $\tilde{u}_\epsilon\in u_\epsilon +\epsilon/\mu \mathbb{B}$ with $-\lambda^T\tilde{u}_\epsilon  +\delta_{S}(\tilde{u}_\epsilon)\leq
  -\lambda^Tu_\epsilon  +\delta_{S}(u_\epsilon)$ and ${\rm arg}\min_u\{ -\lambda^Tu+\delta_S(u)+\mu\|u-\tilde{u}_\epsilon\| \}=\{\tilde{u}_\epsilon\}$. According to the first-order optimality conditions, we have
    \[0\in -\lambda+\partial\big(\delta_S(u)+\mu\|u-\tilde{u}_\epsilon\|\big)|_{u=\tilde{u}_\epsilon}\subseteq -\lambda+\mu {\cal B}+N_S(\tilde{u}_\epsilon),\]
    where the second inclusion follows from the subdifferential sum rule \cite[Corollary 10.9]{RW98}.
    Hence there exists $\lambda_\epsilon'$ such that
    $\|\lambda_\epsilon'-\lambda\|\leq \mu\leq \epsilon$ and $\lambda_\epsilon'\in N_S(\tilde{u}_\epsilon)$.  Note that $\lambda_\epsilon'\to \lambda$ as $\epsilon\to 0$ and
    \begin{equation}\label{2epsi}
   {\|\lambda'_\epsilon-\lambda\| \|\tilde{u}_\epsilon\|\leq}\mu(\|u_\epsilon\|+\epsilon/\mu)=\mu\|u_\epsilon\|+\epsilon\leq 2\epsilon.
    \end{equation}
    Thus
    \begin{eqnarray*}   \hat\sigma_S(\lambda)&=&\liminf_{\tilde\lambda \to\lambda}\inf_u\{\tilde  \lambda^Tu \mid \tilde\lambda\in N_S(u)\}
    \leq \liminf_{\epsilon \to 0^+}\inf_u\{(\lambda_\epsilon')^Tu \mid \lambda'_\epsilon\in N_S(u)\}\\
    &\leq& \liminf_{\epsilon \to 0^+} \langle \lambda'_\epsilon,\tilde{u}_\epsilon\rangle =
    \liminf_{\epsilon \to 0^+} \langle \lambda, \tilde{u}_\epsilon\rangle +\langle \lambda'_\epsilon-\lambda, \tilde{u}_\epsilon\rangle \\
    &\leq & \liminf_{\epsilon \to 0^+}\sigma_S(\lambda)+
      \|\lambda'_\epsilon-\lambda\| \|\tilde{u}_\epsilon\|\\
     &\leq& \liminf_{\epsilon \to 0^+} \sigma_S(\lambda)+2\epsilon\\
     &=& \sigma_S(\lambda),
    \end{eqnarray*}
    where we have used the fact (\ref{2epsi}).

\item By virtue of 1, we only need to prove the inequality $\hat\sigma_S(\lambda)\geq \sigma_S(\lambda)$. Consider an arbitrary $\lambda$. If $\hat\sigma_S(\lambda)=\infty$ there is nothing to show. Hence we can assume  $\hat\sigma_S(\lambda)<\infty$. Then we can find a sequences $\lambda_k$ converging to $\lambda$ such that
 $$   \hat\sigma_S(\lambda)=\liminf_{\tilde\lambda \to\lambda}\inf_u\{\tilde  \lambda^Tu \mid \tilde\lambda\in N_S(u)\}=\lim_{k\rightarrow \infty} \inf_u\{\lambda_k^Tu\mid \lambda_k\in N_S(u)\}.$$
     %  \[\rho_k:=\inf_u\{\lambda_k^Tu\mid \lambda_k\in N_C(u)\}\to \hat\sigma_C(\lambda).\]
       By convexity of $S$ we have by (\ref{supportfunction}) that $\lambda_k^Tu=\sigma_S(\lambda_k)$ whenever $\lambda_k\in N_S(u)$. It follows by the above and    the lower semi-continuity of the support function that
     $$   \hat\sigma_S(\lambda)=\liminf_{\tilde\lambda \to\lambda}\inf_u\{\tilde  \lambda^Tu \mid \tilde\lambda\in N_S(u)\} =\liminf_{
     \lambda_k \to\lambda} \sigma_S(\lambda_k)\geq \sigma_S(\lambda).$$
 \end{enumerate}
\end{proof}
%We now consider the dual form of second-order necessary optimality conditions.
Let the Lagrange function of problem (P) be
$$L(x,\lambda):=f(x)+g(x)^T \lambda.$$
 Consider
the following directional Mordukhovich (M-) multiplier  set:
\begin{eqnarray*}
 \Lambda (\bar x;d) := \{\lambda \, | \, \nabla_x L(\bar x,\lambda)=0,
      \ \lambda\in {N}_{\Lambda}(g(\bar x);\nabla g(\bar x)d)\}.
 \end{eqnarray*}
The following directional first-order necessary optimality condition holds at a local minimizer under the directional metric subregularity.
 \begin{proposition}\label{PropDirFirstOrder}
 Let  $\bar{x}$ be a local optimal solution of problem (P). Suppose that the set-valued map $M(x):=g(x)-\Lambda$ is metrically subregular  at $(\bar{x},0)$ in direction $d$ with $d\in C(\bar{x})$. Then the  directional M-multiplier set $\Lambda (\bar x;d)$ is nonempty.
 \end{proposition}
 \begin{proof}
 By \cite[Theorem 7]{Gfr13a} there is some $\lambda$ satisfying
 \[0\in\nabla f(\bar x)+D^*M((\bar x,0);(d,0))(\lambda)\]
 where the directional limiting coderivative $D^*M((\bar x,0);(d,0))$ is defined by
 \[ u^*\in D^*M((\bar x,0);(d,0))(\lambda) \Longleftrightarrow (u^*,-\lambda )\in N_{{\rm gph\,} M}((\bar x,0);(d,0)).\]
 Since ${\rm gph\,} M=\{(x,y)| \, g(x)-y\in \Lambda\}$, we obtain from \cite[Corollary 3.2]{BeGfrOut19}
 \[N_{{\rm gph\,} M}((\bar x,0);(d,0))\subseteq\{(\nabla g(\bar x)^T\mu, -\mu)\,\mid\,\mu\in N_\Lambda\big(g(\bar x);\nabla g(\bar x)d\big)\}\]
 yielding the assertion of the proposition.
 \end{proof}

In the following theorem we give a second-order necessary optimality condition for problem (P) in terms of directional M-multipliers under the directional metric subregularity condition.

   \begin{theorem}\label{second-order-theorem-1}
Let  $\bar{x}$ be a local optimal solution of problem (P). Suppose that the set-valued map $M(x):=g(x)-\Lambda$ is metrically subregular at $(\bar{x},0)$ in direction $d$ with $d\in C(\bar{x})$ and $T^2_{\Lambda}\big(g(\bar{x});\nabla g(\bar{x})d\big)\neq \emptyset$.
 %and $T^2_{\Lambda}\big(g(\bar x);\nabla g(\bar x)d \big)\neq \emptyset$.
  Then there exists a directional M-multiplier $\lambda\in \Lambda (\bar x;d)$ such that for every set $A$ satisfying $ A\supseteq \nabla g(\bar x) T_{\cal F}^2(\bar x;d)+\nabla^2g(\bar x)(d,d)$
  %$ A\supseteq \nabla g(\bar x) \R^n+\nabla^2g(\bar x)(d,d)$
  one has
% \begin{equation}\label{multiplier-2-new}
% 0=\nabla f(\bar x) +\nabla g(\bar x)^T\lambda, \quad \lambda\in  N_{\Lambda}(g(\bar x);\nabla g(\bar x)d)
% \end{equation}
%and
\begin{equation}\label{EqSecOrderOpt}
  \nabla^2_{xx}L(\bar{x},\lambda)(d,d)-\hat\sigma_{T_\Lambda^2(g(\bar x);\nabla g(\bar x)d ),A}(\lambda)\geq 0.
\end{equation}
In particular, there exists a multiplier $\lambda\in \Lambda (\bar x;d)$ such that  \begin{equation}\label{EqSecOrderOptnew}
  \nabla^2_{xx}L(\bar{x},\lambda)(d,d)-\hat\sigma_{T^2_{\Lambda}(g(\bar x);\nabla g(\bar x)d ) }(\lambda)\geq 0.
\end{equation}
 \end{theorem}

  \begin{proof} Note that by Corollary \ref{CorNonEmpty}, under the assumptions of the theorem we have $T_{\cal F}^2(\bar x;d)\not =\emptyset$.
  % Moreover by   Proposition \ref{PropSecOrderTanSet}, under the assumption of the theorem we have
  %$$ T_\Lambda^2(g(\bar x);\nabla g(\bar x)d )\supseteq\nabla g(\bar x) T_{\cal F}^2(\bar x;d)+\nabla^2g(\bar x)(d,d)  .
  %$$ Hence one can find a nonempty set $A$ satisfying the required inclusion.
  Consider for every $\epsilon>0$ the optimization problem
  \begin{equation} \min_w \nabla f(\bar x)w+\frac \epsilon2\norm{w}^2\mbox{ subject to }w\in T^2_{\mathcal F}(\bar x;d).\label{op1}
  \end{equation}
  Since $\epsilon>0$,  problem (\ref{op1})  has a globally optimal solution $w_\epsilon$ because the objective is coercive and the set $T^2_{\mathcal F}(\bar x;d)$ is closed and nonempty. We claim that $\lim_{\epsilon\downarrow 0}\epsilon\norm{w_\epsilon}=0$. Indeed,  by Theorem \ref{ThBasicSecOrd} together with the optimality of $w_\epsilon$ we obtain
  \[-\nabla^2 f(\bar x)(d,d)+\frac \epsilon2\norm{w_\epsilon}^2\leq \nabla f(\bar x)w_\epsilon+\frac \epsilon2\norm{w_\epsilon}^2\leq \nabla f(\bar x)\bar w+\frac \epsilon2\norm{\bar w}^2\]
  for arbitrarily chosen  $\bar w\in T^2_{\mathcal F}(\bar x;d)$, yielding
  \[
 \epsilon\|w_\epsilon\|\leq \sqrt{2 \epsilon}\Big(\nabla f(\bar{x})\bar w+\nabla^2f(\bar{x})(d,d)+\frac{\epsilon}{2}\|\bar w\|^2\Big)^{\frac{1}{2}}.
 \]
 Taking the limit on the both sides of the above inequality yields $\epsilon\|w_\epsilon\|\to 0$ as $\epsilon\downarrow 0$.
 By \cite[Theorem 6.12]{RW98}, the basic first-order optimality condition for problem (\ref{op1}) at $w_\epsilon$
 \begin{equation}
 -\nabla f(\bar x)-\epsilon w_\epsilon\in N_{T^2_{\mathcal F}(\bar x;d)}(w_\epsilon)\label{basicOC}\end{equation}
 is fulfilled. By Proposition \ref{PropSecOrderTanSet},
 $T^2_{\mathcal{F}}(\bar x;d)=\{w\,|\,  P(w)\in D\},$
 where $P(w):=\nabla g(\bar x)w+\nabla^2g(\bar x)(d,d)$, $D:= T^2_\Lambda\big(g(\bar x);\nabla g(\bar x)d\big)$, and MSCQ holds at  $w_\epsilon\in  T^2_{\mathcal{F}}(\bar x;d) $ for the system $P(w)\in D$ with modulus $\kappa$ which is  the modulus of metric subregularity of $M$ at $(\bar x,0)$ in direction $d$. It follows by
 \cite[Theorem 3]{GY17} that
 $$N_{T^2_{\mathcal F}(\bar x;d)}(w_\epsilon) \subseteq \{z| \exists \lambda_\epsilon \in \kappa \|z\| {\cal B} \cap N_D(P(w_\epsilon)) \mbox{ with } z=\nabla P(w_\epsilon)^T\lambda_\epsilon \}.$$
By virtue of  (\ref{basicOC}) and the above inclusion,  there is some multiplier
\begin{equation}
\lambda_\epsilon\in \kappa\norm{\nabla f(\bar x)+\epsilon w_\epsilon}{\mathcal B}\cap N_{T^2_\Lambda(g(\bar x);\nabla g(\bar x)d)}\big(\nabla g(\bar x) w_\epsilon+\nabla ^2g(\bar x)(d,d)\big)\label{multiplier}
\end{equation}  such that
 \begin{equation}\label{Eqfirstordercondition}\nabla f(\bar x)+\epsilon w_\epsilon+\nabla g(\bar x)^T\lambda_\epsilon=0.\end{equation}
  Since $\epsilon w_\epsilon\to 0$ as shown above,  $\lambda_\epsilon$ is bounded as $\epsilon$ sufficiently small.
   Hence we can take a sequence of positive numbers $\epsilon_k$ converging to $0$ such that the corresponding sequence of multipliers $\lambda_{\epsilon_k}$ converges to some $\lambda$.
    Taking limits as $\epsilon_k \to 0$ in \eqref{Eqfirstordercondition} we obtain
    $$\nabla f(\bar x)+\nabla g(\bar x)^T\lambda=0.$$
    By (\ref{multiplier}) and Lemma \ref{relationship-normal-cone}, we have
    $$\lambda_\epsilon\in  N_{T^2_\Lambda(g(\bar x);\nabla g(\bar x)d)}\big(\nabla g(\bar x) w_\epsilon+\nabla ^2g(\bar x)(d,d)\big)\subseteq N_\Lambda (g(\bar x); \nabla g(\bar x)d).$$ Taking limits as $\epsilon_k \to 0$, we obtain $\lambda \in N_\Lambda (g(\bar x); \nabla g(\bar x)d)$ and consequently
    $\lambda\in \Lambda(\bar x;d)$.

    Now consider any set $A\supseteq\nabla g(\bar x)T_{\cal F}^2(\bar x;d)+\nabla ^2g(\bar x)(d,d)$. Setting $u_k:=\nabla g(\bar x) w_{\epsilon_k}+\nabla ^2g(\bar x)(d,d)$ we have $\lambda_{\epsilon_k}\in N_{T^2_\Lambda(g(\bar x);\nabla g(\bar x)d)}(u_k) $ and $u_k\in A$. Taking into account (\ref{Eqfirstordercondition}) and
     $\nabla f(\bar x)w_{\epsilon_k}+\nabla^2f(\bar x)(d,d)\geq0$ by virtue of  Theorem \ref{ThBasicSecOrd} we obtain
    \begin{eqnarray*}\nabla_{xx}^2L(\bar x,\lambda)(d,d)&=&\lim_{k\to\infty}\left \{ \nabla_{xx}^2L(\bar x,\lambda_{\epsilon_k})(d,d)+\big(\nabla f(\bar x)+\epsilon_k w_{\epsilon_k}+\nabla g(\bar x)^T\lambda_{\epsilon_k} )^Tw_{\epsilon_k} \right \}\\
    &=&\lim_{k\to\infty}\Big(\nabla f(\bar x)w_{\epsilon_k}+\nabla^2f(\bar x)(d,d)+\lambda_{\epsilon_k}^Tu_k+\epsilon_k\norm{w_{\epsilon_k}}^2\Big)\\
    &\geq& \limsup_{k\to\infty}\lambda_{\epsilon_k}^Tu_k\geq \limsup_{k\to\infty}\inf \{\lambda_{\epsilon_k}^Tu\,\mid\, u\in N_{T^2_{\Lambda}(g(\bar x);\nabla g(\bar x)d )}^{-1}(\lambda_{\epsilon_k})\cap A\}\\
    &\geq& \liminf_{\tilde\lambda\to\lambda}\inf \{\tilde\lambda^Tu\,\mid\, u\in N_{T^2_{\Lambda}(g(\bar x);\nabla g(\bar x)d )}^{-1}(\tilde\lambda)\cap A\}\\
    &=&\hat\sigma_{T^2_{\Lambda}(g(\bar x);\nabla g(\bar x)d ),A}(\lambda).
    \end{eqnarray*}
% \red{By Proposition \ref{PropSecOrderTanSet}, under the assumption of the theorem we have
% $$ T_\Lambda^2(g(\bar x);\nabla g(\bar x)d )\supseteq\nabla g(\bar x) T_{\cal F}^2(\bar x;d)+\nabla^2g(\bar x)(d,d)  .
% $$
   In particular if we take $A=\Re^m$, then  $\hat\sigma_{T^2_{\Lambda}(g(\bar x);\nabla g(\bar x)d ),A}=\hat \sigma_{T^2_{\Lambda}(g(\bar x);\nabla g(\bar x)d )}$ and hence (\ref{EqSecOrderOptnew}) holds.
   \end{proof}
 Among all possible choices for $A$, the second-order necessary condition \eqref{EqSecOrderOpt} is strongest for $A=A_{\rm opt}:=\nabla g(\bar x) T_{\cal F}^2(\bar x;d)+\nabla^2g(\bar x)(d,d) $
  %\nabla g(\bar x)\Re^n+\nabla^2g(\bar x)(d,d)$
  and weakest for  $A=\Re^m$.  There is also an intermediate choice of $A_{\rm mid}:=\nabla g(\bar x)\Re^n+\nabla^2g(\bar x)(d,d)$. %However choosing the optimal one
  {The optimal choice} $A_{\rm opt}$ involves the second-order tangent cone $T_{\cal F}^2(\bar x;d)$ which  is in general hard to compute and, moreover, if the second-order cone $T_{\cal F}^2(\bar x;d)$ is known we may use the primal optimality condition Theorem \ref{ThBasicSecOrd} instead.
   On the other hand, choosing $A=\R^m$ results in the weakest optimality condition but more trackable
   %$\hat \sigma_{T^2_{\Lambda}(g(\bar x);\nabla g(\bar x)d ),A_{\rm opt} }$  might be harder to calculate than
    lower generalized support function $\hat \sigma_{T^2_{\Lambda}(g(\bar x);\nabla g(\bar x)d )}$. The intermediate choice $A_{\rm mid}:=\nabla g(\bar x)\Re^n+\nabla^2g(\bar x)(d,d)$ may result in stronger optimality conditions and  slightly harder to calculate $\hat \sigma_{T^2_{\Lambda}(g(\bar x);\nabla g(\bar x)d ),A_{\rm mid}}$  than $\hat \sigma_{T^2_{\Lambda}(g(\bar x);\nabla g(\bar x)d )}$.

Note that in Theorem  \ref{second-order-theorem-1}, even the first order optimality condition is stronger than the classical M-stationary condition since the directional limiting normal cone  is in general smaller than the nondirectional limiting normal cone.
However in the case where $\Lambda$ is convex, the directional M-stationary condition in a critical direction coincides with the classical stationary condition
{and the directional M-multiplier set $\Lambda(\bar x;d)$ coincides with the classical multiplier set
\[\Lambda(\bar x):=\{\lambda\in N_\Lambda(g(\bar x))\, |\, \nabla_x L(\bar x,\lambda)=0\}\]
for every critical direction $d\in C(\bar x)$. Indeed, the inclusion $\Lambda(\bar x;d)\subseteq\Lambda(\bar x)$ obviously holds. Now pick any $\lambda\in \Lambda(\bar x)$. Since $\nabla g(\bar x)d\in T_\Lambda(g(\bar x))$ and $\lambda\in N_\Lambda(g(\bar x))$, we have
\[0\geq \lambda^T\nabla g(\bar x)d=-\nabla f(\bar x)d\geq 0\]
implying $\lambda \in  \{\nabla g(\bar x)d\}^\perp$. Owing to \eqref{EqDirLimNormalConeConvex} we conclude $\lambda\in N_\Lambda(g(\bar x);\nabla g(\bar x)d)$ and $\lambda\in\Lambda(\bar x;d)$ follows.}
We now specialize Theorem   \ref{second-order-theorem-1} under the additional assumption that $\Lambda$ is convex and $T^2_{\Lambda}(g(\bar x);\nabla g(\bar x)d )$ is convex. In this case, we obtain the same second order necessary optimality as  the classical result of \cite[Theorem 3.45]{BS} under the directional metric subregularity which is weaker than the Robinson's constraint qualification.
\begin{corollary} Let  $\bar{x}$ be a local optimal solution of problem (P) where $\Lambda$ is convex. Suppose that the set-valued map $M(x):=g(x)-\Lambda$ is metrically subregular   at $(\bar{x},0)$ in direction $d$ with $d\in C(\bar{x})$.
% and $T^2_{\Lambda}\big(g(\bar{x});\nabla g(\bar{x})d\big)\neq \emptyset$.
 %and $T^2_{\Lambda}\big(g(\bar x);\nabla g(\bar x)d \big)\neq \emptyset$.
%  Then there exists a multiplier $\lambda\in \Lambda (\bar x):=\{\lambda|\nabla_x L(\bar x, \lambda)=0, \lambda \in N_\Lambda(g(\bar x)\}$ such that for every convex set $A$ satisfying $T_\Lambda^2(g(\bar x);\nabla g(\bar x)d )\supseteq A\supseteq \nabla g(\bar x) T_{\cal F}^2(\bar x;d)+\nabla^2g(\bar x)(d,d)$ one has
%% \begin{equation}\label{multiplier-2-new}
%% 0=\nabla f(\bar x) +\nabla g(\bar x)^T\lambda, \quad \lambda\in  N_{\Lambda}(g(\bar x);\nabla g(\bar x)d)
%% \end{equation}
%%and
%\begin{equation*}
%  \nabla^2_{xx}L(\bar{x},\lambda)(d,d)-\sigma_{A}(\lambda)\geq 0.
%\end{equation*}
If  the second-order tangent cone $T^2_{\Lambda}(g(\bar x);\nabla g(\bar x)d )$ is convex, then there exists a multiplier $\lambda\in \Lambda (\bar x)$ such that  \begin{equation*}
  \nabla^2_{xx}L(\bar{x},\lambda)(d,d)-\sigma_{T^2_{\Lambda}(g(\bar x);\nabla g(\bar x)d ) }(\lambda)\geq 0.
\end{equation*}
\end{corollary}
{\begin{proof} If $T^2_{\Lambda}\big(g(\bar{x});\nabla g(\bar{x})d\big)= \emptyset$, then  $\sigma_{T^2_{\Lambda}(g(\bar x);\nabla g(\bar x)d ) }(\lambda)=-\infty$ by convention and so there is nothing to prove.
The case of $T^2_{\Lambda}\big(g(\bar{x});\nabla g(\bar{x})d\big)\neq \emptyset$ follows from \eqref{EqSecOrderOptnew} and Proposition \ref{PropHatSigma}(2).
\end{proof}}
%\begin{proof} \red{If $T_\Lambda^2(g(\bar x);\nabla g(\bar x)d )=\emptyset$, then by definition its support function is equal to negative  infinity and hence there is nothing to prove. Now suppose that $T_\Lambda^2(g(\bar x);\nabla g(\bar x)d )\not = \emptyset$. By Proposition \ref{PropSecOrderTanSet}, under the metric subregularity condition, we have
%   $$ T_\Lambda^2(g(\bar x);\nabla g(\bar x)d )\supseteq\nabla g(\bar x) T_{\cal F}^2(\bar x;d)+\nabla^2g(\bar x)(d,d).
% $$
% and hence one can find a convex set $A$ satisfying the condition
% $$T_\Lambda^2(g(\bar x);\nabla g(\bar x)d )\supseteq A\supseteq \nabla g(\bar x) T_{\cal F}^2(\bar x;d)+\nabla^2g(\bar x)(d,d).$$ Since the set $A$ is convex and is included in $T_\Lambda^2(g(\bar x);\nabla g(\bar x)d )$, it follows that
% $$
% \hat\sigma_{T^2_{\Lambda}(g(\bar x);\nabla g(\bar x)d ),A}(\lambda)
%    \geq  \red{\hat\sigma_{A}}(\lambda)=\sigma_A(\lambda).$$
%    Now apply Theorem \ref{second-order-theorem-1} we obtain the desired result taking into account that the directional M-multiplier set $\Lambda(\bar x;d)=\Lambda(\bar x)$ for any $d\in C(\bar x)$ as shown in the discussion before the statement of this corollary.
% }
% \end{proof}

Consider the following directional Clarke (C-) multiplier set:
 $${\Lambda^c(\bar x;d)} := \{\lambda \, | \, \nabla_x L(\bar x,\lambda)=0,
      \ \lambda\in N^c_{\Lambda}(g(\bar x);\nabla g(\bar x)d)\}.$$
      It is clear that the set of directional C-multipliers is {closed convex  and} in general larger than the set of directional M-multipliers.

      In what follows, we derive a second-order necessary optimality condition for problem (P) in terms of directional C-multipliers under the  constraint qualification condition \begin{equation}\label{EqClMetrReg}
        \nabla g(\bar{x})^T\lambda=0, \ \ \lambda\in {N^c_\Lambda(g(\bar{x}); \nabla g(\bar x)d) }\ \Longrightarrow\ \lambda=0
      \end{equation}
      {which we will call {\em directional Robinson's constraint qualification} (DirRCQ) in direction $d$.}
      Condition \eqref{EqClMetrReg} is stronger than FOSCMS in direction $d$, under which  by virtue of Proposition \ref{FOSCMS} the constraint mapping $M(x)=g(x)-\Lambda$ is metrically subregular in direction $d$.
\begin{lemma}\label{LemClMetrReg}
The following three statements are equivalent:
\begin{enumerate}
\item[{ (i)}]{DirRCQ in direction $d$ holds.}%Condition \eqref{EqClMetrReg} holds.
\item[{ (ii)}]
\begin{equation}\label{EqRobCQ}
 \nabla g(\bar{x})\Re^n+\widehat{T}_{\Lambda}(g(\bar{x});\nabla g(\bar{x})d)=\Re^m.
\end{equation}
\item[{ (iii)}] The set $\Lambda^c(\bar x;d)$ is compact, whenever it is nonempty.
\end{enumerate}
\end{lemma}
\begin{proof}Condition \eqref{EqClMetrReg} can be equivalently written as
\begin{equation}\label{EqClMetrReg1}
    \ker \nabla g(\bar x)^T\cap N^c_\Lambda(g(\bar{x}); \nabla g(\bar x)d)=\{0\}.
\end{equation}
By taking polars on  both sides of the equation and using the rule for polar cones \cite[Corollary 11.25]{RW98} and the fact that $\big(N^c_{\Lambda}(g(\bar{x});\nabla g(\bar{x})d)\big)^\circ=\widehat{T}_{\Lambda}(g(\bar{x});\nabla g(\bar{x})d)$, we can see that condition  \eqref{EqClMetrReg} is equivalent to saying
 $\cl\Big(\nabla g(\bar{x})\Re^n+\widehat{T}_{\Lambda}(g(\bar{x});\nabla g(\bar{x})d)\Big)=\Re^m.$ Obviously the set $\nabla g(\bar{x})\Re^n+\widehat{T}_{\Lambda}(g(\bar{x});\nabla g(\bar{x})d)$ is convex and thus $\ri(\nabla g(\bar{x})\Re^n+\widehat{T}_{\Lambda}(g(\bar{x});\nabla g(\bar{x})d))=\Re^m$ by \cite[Theorem 6.3]{Ro70}. It follows that condition \eqref{EqRobCQ} holds and thus the implication``(i)$\Rightarrow$(ii)'' is established. In order to show the reverse implication, just note that by taking polars on both sides of \eqref{EqRobCQ} we obtain \eqref{EqClMetrReg1}. Finally, the equivalence between (i) and (iii) follows from \cite[Theorem 8.4]{Ro70} together with the fact that the recession cone to $\Lambda^c(\bar x;d)$ is exactly the set on the left hand side of \eqref{EqClMetrReg1}.
\end{proof}
\begin{proposition}\label{PropSecOrdClarke}
Let $\bar x$ be feasible for the problem (P). Suppose that $d\in C(\bar{x})$ satisfies
  $T^2_{\Lambda}(g(\bar{x});\nabla g(\bar{x})d)\neq \emptyset$ and {DirRCQ.} %condition \eqref{EqClMetrReg}.
  Then the following three statements are equivalent:
  \begin{enumerate}
  \item[(i)] The primal second-order necessary condition
  \[\nabla f(\bar x)w+\nabla^2f(\bar x)(d, d) \geq 0, \quad \forall w\in T^2_{\cal F}(\bar{x};d)\]
  of Theorem \ref{ThBasicSecOrd} holds.
  \item[(ii)] For every $u\in T^2_{\Lambda}(g(\bar{x});\nabla g(\bar{x})d)$, there exists $\lambda_u\in \Lambda^c(\bar{x};d)$ such that
  \[
   \nabla^2_{xx}L(\bar{x},\lambda_u)(d,d)-\lambda_u^T u \geq 0.
  \]
  \item[(iii)] For every  nonempty convex subset $C\subseteq T^2_{\Lambda}(g(\bar{x});\nabla g(\bar{x})d)$, there exists $\lambda\in \Lambda^c(\bar{x};d)$ such that
  \[
   \nabla^2_{xx}L(\bar{x},\lambda)(d,d)-\sigma_C(\lambda) \geq 0.
  \]
  \end{enumerate}
\end{proposition}
 \begin{proof}
Since condition \eqref{EqClMetrReg} implies the metric subregularity  in direction $d$, by Proposition \ref{PropSecOrderTanSet} we have
\begin{equation}  w\in T^2_{\cal F}(\bar x;d) \Longleftrightarrow  \nabla  g(\bar x) w+{\nabla}^2 g(\bar x)(d,d)\in T^2_{\Lambda}(g(\bar x);\nabla g(\bar{x})d). \label{equivalence} \end{equation}
We first show the implication ``(i)$\Rightarrow$(ii)''.
Take $u\in T^2_\Lambda(g(\bar{x});\nabla g(\bar{x})d)$. Then $$u+\widehat{T}_\Lambda(g(\bar{x});{\nabla g(\bar x)}d)\subseteq T^2_{\Lambda}(g(\bar{x});\nabla g(\bar{x})d)$$ by virtue of Proposition \ref{PropRegTanCone}. Since $ \nabla f(\bar{x})w+\nabla^2f(\bar{x})(d,d)\geq 0$ for all $w \in T^2_{\cal F}\big(\bar x; d \big)$, the following conic linear program
   \begin{flalign}\label{conicp}
 \min_w \ \ & \nabla f(\bar{x})w+\nabla^2f(\bar{x})(d,d) \nonumber \\
 \mbox{s.t.}  \ \ & \nabla g(\bar{x})w+\nabla^2g(\bar{x})(d,d)\in u+\widehat{T}_{\Lambda}(g(\bar{x});\nabla g(\bar{x})d)
 \end{flalign}
has  nonnegative optimal value. The dual program of the  conic linear program (\ref{conicp}) is
 \begin{flalign*}
 \max_{\lambda\in (\widehat{T}_{\Lambda}(g(\bar{x});\nabla g(\bar{x})d))^\circ} \ \ &
 \nabla^2_{xx}L(\bar{x},\lambda)(d,d)-\lambda^Tu \\
 s.t.  \ \ & \nabla f(\bar{x})+\nabla g(\bar{x})^T\lambda=0.
 \end{flalign*}
Since by Proposition \ref{polarity},
 $(\widehat{T}_{\Lambda}(g(\bar{x});\nabla g(\bar{x})d))^\circ=N^c_{\Lambda}(g(\bar{x});\nabla g(\bar{x})d)$, the above dual problem can be equivalently written as
 \[
 \max_{\lambda\in\Lambda^c(\bar{x};d)} \{\nabla^2_{xx}L(\bar{x},\lambda)(d,d)-\lambda^Tu\}.
 \]
By Lemma \ref{LemClMetrReg}, condition \eqref{EqClMetrReg} is equivalent to \eqref{EqRobCQ} and it is easy to see that the latter implies
 $$ 0\in {\rm int} \{\nabla g(\bar{x})\Re^n+\nabla^2g(\bar{x})(d,d)-\widehat{T}_{\Lambda}(g(\bar{x});\nabla g(\bar{x})d)\},$$
cf. \cite[(2.354)]{BS}.
Consequently by \cite[Theorem 2.187]{BS}, there is no dual gap and the dual program has an optimal solution $\lambda_u$ such that
\[ \max_{\lambda\in\Lambda^c(\bar{x};d)}\{\nabla^2_{xx}L(\bar{x},\lambda)(d,d)-\lambda^Tu\}=\nabla^2_{xx}L(\bar{x},\lambda_u)(d,d)-\lambda_u^Tu\geq 0.\]
This proves ``(i)$\Rightarrow$(ii)''.

In order to show the reverse implication, take $w\in T^2_{\mathcal{F}}(\bar x;d)$ together with $u:=\nabla  g(\bar x) w+{\nabla}^2 g(\bar x)(d,d)\in T^2_\Lambda(g(\bar x);\nabla g(\bar x)w)$, where the existence of such $w$ is guaranteed by virtue of (\ref{equivalence}),  and $\lambda_u\in\Lambda^c(\bar x;d)$ fulfilling $ \nabla^2_{xx}L(\bar{x},\lambda_u)(d,d)-\lambda_u^T u \geq 0$. Then
\[\nabla f(\bar x)w+\nabla^2f(\bar x)(d,d)=-\lambda_u^T\nabla g(\bar x)w+\nabla^2f(\bar x)(d,d)=\nabla^2_{xx}L(\bar{x},\lambda_u)(d,d)-\lambda_u^T u \geq 0\]
showing ``(ii)$\Rightarrow$(i)''.\\
Since ``(iii)$\Rightarrow$(ii)'' always hold, there remains to show ``(ii)$\Rightarrow$(iii)''. Consider a nonempty convex subset $C\subseteq T^2_{\Lambda}(g(\bar{x});\nabla g(\bar{x})d)$. Since $\sigma_C=\sigma_{\cl C}$, we may assume that $C$ is closed.
For every $u\in C$ the corresponding $\lambda_u$ according to (ii) fulfills $\lambda_u\in\Lambda^c(\bar x;d)$ and hence $\Lambda^c(\bar x;d)$ is nonempty. We conclude from Lemma \ref{LemClMetrReg} that $\Lambda^c(\bar x;d)$ is compact and therefore
by the minimax theorem  in \cite[Corollary 37.3.2]{Ro70} we have
\begin{eqnarray}\label{EqAux1}\inf_{u\in C}\sup_{\lambda\in \Lambda^c(\bar x;d)}\nabla^2_{xx}L(\bar{x},\lambda)(d,d)-\lambda^Tu&=&\sup_{\lambda\in \Lambda^c(\bar x;d)}\inf_{u\in C}\nabla^2_{xx}L(\bar{x},\lambda)(d,d)-\lambda^Tu\\
\label{EqAux2}&=&\sup_{\lambda\in \Lambda^c(\bar x;d)}\nabla^2_{xx}L(\bar{x},\lambda)(d,d)-\sigma_C(\lambda).\end{eqnarray}
 Due to (ii) the quantity on left hand side of \eqref{EqAux1} is nonnegative. On the other hand, the supremum in \eqref{EqAux2} is attained at some $\lambda$, since $\Lambda^c(\bar x;d)$ is compact and $-\sigma_C(\cdot)$ is upper semi-continuous. This completes the proof.
\end{proof}
The following second-order necessary optimality condition follows immediately by Theorem \ref{ThBasicSecOrd} and Proposition \ref{PropSecOrdClarke}.
\begin{corollary}\label{CorClSecondOrder}
    Let $\bar{x}$ be a local optimal solution of problem (P).  Suppose that $d\in C(\bar{x})$ satisfies
  $T^2_{\Lambda}(g(\bar{x});\nabla g(\bar{x})d)\neq \emptyset$ and {DirRCQ.} %condition \eqref{EqClMetrReg} is fulfilled.
  Then, for every nonempty convex subset $C\subseteq T^2_{\Lambda}(g(\bar{x});\nabla g(\bar{x})d)$ there is some  $\lambda\in \Lambda^c(\bar{x};d)$ such that
  \[
   \nabla^2_{xx}L(\bar{x},\lambda)(d,d)-\sigma_C(\lambda) \geq 0.
  \]
\end{corollary}
A close look at the proof of Theorem \ref{second-order-theorem-1} shows that the second-order necessary conditions stated therein are implied by the primal second-order necessary condition of Theorem \ref{ThBasicSecOrd}. Hence, in view of Proposition \ref{PropSecOrdClarke}, the second-order necessary condition of Corollary \ref{CorClSecondOrder} are stronger than the one of Theorem \ref{second-order-theorem-1}. However, the constraint qualification {DirRCQ} %\eqref{EqClMetrReg}
 used in Corollary \ref{CorClSecondOrder} is also stronger than the one of Theorem \ref{second-order-theorem-1}.

We now want to compare Corollary \ref{CorClSecondOrder} with the classical result of \cite[Theorem 3.45]{BS} under the additional assumption that $\Lambda$ is convex.
{
By \eqref{EqDirLimNormalConeConvex} and using the convexity of $N_\Lambda(g(\bar x);\nabla g(\bar x)d)$, for every $d$ with $\nabla g(\bar x)d\in T_\Lambda(g(\bar x))$ there holds
\begin{eqnarray*}\nabla g(\bar x)^T\lambda=0,\ \lambda\in N_\Lambda^c(g(\bar x);\nabla g(\bar x)d)&\Leftrightarrow&  \nabla g(\bar x)^T\lambda=0,\ \lambda\in N_\Lambda(g(\bar x)),\ \lambda\in\{\nabla g(\bar x)d\}^\perp\\
&\Leftrightarrow&\nabla g(\bar x)^T\lambda=0,\ \lambda\in N_\Lambda(g(\bar x))\\
&\Leftrightarrow& \lambda\in (\nabla g(\bar x)\R^n)^\perp\cap N_\Lambda(g(\bar x)).
\end{eqnarray*}
Thus, by \cite[Proposition 2.97]{BS} the directional Robinson's constraint qualification %constraint qualification \eqref{EqClMetrReg}
 is equivalent to {the non-directional one} %Robinson's constraint qualification
\[\nabla g(\bar x)\Re^n+T_\Lambda(g(\bar x))=\Re^m.\]
Since we also have $\Lambda^c(\bar x;d)=\Lambda(\bar x)$ as pointed out above, in case of convex $\Lambda$ the Corollary \ref{CorClSecondOrder} is equivalent with \cite[Theorem 3.45]{BS}.
}

\if{
Since  $\Lambda$ is convex, by \cite[Corollary 6.29]{RW98}, for every $d$ with $\nabla g(\bar x)d\in T_\Lambda(g(\bar x))$ we have
\[T_\Lambda(g(\bar x))=\widehat T_\Lambda(g(\bar x))\subseteq \widehat T_\Lambda(g(\bar x);\nabla g(\bar x)d)\subseteq T_\Lambda(g(\bar x))\]
implying $\widehat T_\Lambda(g(\bar x);\nabla g(\bar x)d)= T_\Lambda(g(\bar x))$ and consequently $N^c_\Lambda((g(\bar x);\nabla g(\bar x)d))=N_\Lambda(g(\bar x))$. Thus, by Lemma \ref{LemClMetrReg} the constraint qualification \eqref{EqClMetrReg} is equivalent to Robinson's constraint qualification
\[\nabla g(\bar x)\Re^n+T_\Lambda(g(\bar x))=\Re^m\]
and the directional C-multiplier set coincides with the classical multiplier set $\Lambda(\bar x)=\{\lambda\in N_\Lambda(g(\bar x))\;\mid\;\nabla_x L(\bar x,\lambda)=0\}$.
Hence, in case of convex $\Lambda$ the Corollary \ref{CorClSecondOrder} is equivalent with \cite[Theorem 3.45]{BS}.
 }\fi

When the directional C-multiplier set $\Lambda^c(\bar x;d)=\{\lambda_0\}$ is a singleton,
it is easy to see from Proposition \ref{PropSecOrdClarke}(ii) and Theorem \ref{ThBasicSecOrd} that
\begin{equation}\label{EqSecOrdSingleton}\nabla_{xx}^2L(\bar x,\lambda_0)(d,d)-\sigma_{T^2_\Lambda(g(\bar x);\nabla g(\bar x)d)}(\lambda_0)\geq 0\end{equation}
is a necessary second-order condition at a local minimizer $\bar x$. Note that by definition $\sigma_{T^2_\Lambda(g(\bar x);\nabla g(\bar x)d)}(\lambda_0)=-\infty$ whenever $T^2_\Lambda(g(\bar x);\nabla g(\bar x)d)=\emptyset$. In this case the above second-order optimality condition holds automatically.
We now try to enhance the condition \eqref{EqClMetrReg} so that the directional C-multiplier  set $\Lambda^c(\bar{x};d)$ is a singleton. As we will show below this is achieved by the directional non-degeneracy condition
 \begin{equation}\label{EqNondegeneracy}
  \nabla g(\bar{x})^T\lambda=0, \ \lambda\in { {\rm span\;}N_{\Lambda}(g(\bar{x});\nabla g(\bar x)d)}\ \Longrightarrow\ \lambda=0
  \end{equation}
 defined for $\bar x$ feasible for the problem (P) and direction $d$ satisfying $\nabla g(\bar x)d\in T_\Lambda(g(\bar x))$, where ${\rm span\;}S$ denotes the affine hull of the set $S$.

 Recall that  by \cite[Definition 6]{GY17} we call a subspace $L$ the generalized linearity space of set $C$ and denote it by  ${\cal L}(C)$ provided that it is   the largest subspace $L$  such that $C +L\subseteq C$.  Note that when $C$ is a convex set,  the generalized linearity space is reduced to the linearity space (\cite[page 65]{Ro70}) and in the case when $C$ is a closed convex cone,  we have ${\cal L}(C)=(-C)\cap C$.

 Given $d\in C(\bar x)$, define the   set of {\em strong multipliers in direction $d$} as
\[\Lambda^s(\bar x;d):=\{\lambda\in \widehat N_{T_\Lambda(g(\bar x))}(\nabla g(\bar x)d)\,|\,\nabla_x L(\bar x,\lambda)=0\}.\]
By \eqref{eq:1new} we have $\Lambda^s(\bar x;d)\subseteq\Lambda(\bar x;d)$.
\begin{lemma}\label{LemDirS_stat}
  Let $\bar{x}$ be a local optimal solution of problem (P) and suppose that the directional non-degeneracy condition \eqref{EqNondegeneracy} holds for a critical direction $d\in C(\bar x)$. Then
  \[\Lambda^s(\bar x;d)=\Lambda(\bar x;d)=\Lambda^c(\bar x;d)=\{\lambda_0\}\]
  is a singleton.
\end{lemma}
\begin{proof}
Since $d\in C(\bar x)$ implies that $\nabla g(\bar x) d\in T_\Lambda(g(\bar x))$ and \eqref{EqNondegeneracy} implies that
  \[\nabla g(\bar{x})^T\lambda=0,\ \lambda\in N_\Lambda(g(\bar{x}); \nabla g(\bar x)d) \Longrightarrow \lambda=0,\]
  FOSCMS for direction $d$ holds and by Proposition \ref{FOSCMS}, the set-valued map $g(x)-\Lambda$ is metrically subregular at $(\bar x, 0)$ in direction $d$. By definition there are reals $\kappa,\rho,\delta>0$ such that
  \begin{equation}
  \dist(x,\mathcal{F})\leq \kappa \dist(g(x),\Lambda)\ \forall x\in\bar x+V_{\rho,\delta}(d).\label{directionaleb} \end{equation}
  We now show by contradiction that $d$ is a locally optimal solution of the program
  \begin{equation}\label{EqLinProgr}\min_{d'} \nabla f(\bar x)d'\quad\mbox{subject to}\quad \nabla g(\bar x)d'\in T_\Lambda(g(\bar x)).\end{equation}
  Assume on the contrary that there is a sequence $d_k\to d$ satisfying $\nabla g(\bar x)d_k\in T_\Lambda(g(\bar x))$ and $\nabla f(\bar x)d_k<0=\nabla f(\bar x)d$. Then we can find some index $\bar k$ and some $\bar t>0$ such that  $\norm{\bar t d_{\bar k}}<\rho$ and
  \[\big\Vert \Vert d\Vert d_{\bar k}-\Vert d_{\bar k}\Vert d\big\Vert\leq \big\Vert \Vert d\Vert d_{\bar k}-\Vert d\Vert d\big\Vert+ \big | \norm{d_{\bar k}}-\norm{d}\big |\norm{d}\leq 2\norm{d_{\bar k}-d}\norm{d}\leq\delta\norm{d_{\bar k}}\norm{d}\]
  implying $td_{\bar k}\in \V_{\rho,\delta}(d)       $ $\forall t\in[0,\bar t]$. By (\ref{directionaleb}),
  there is some sequence $t_n\downarrow 0$ such for all $t_n<\bar t$ we can find some $x_n\in\mathcal{F}$ satisfying
  \[\norm{x_n-(\bar x+t_nd_{\bar k})}\leq\kappa \dist(g(\bar x+t_n d_{\bar k}),\Lambda)=\kappa\big(\dist(g(\bar x)+t_n\nabla g(\bar x)d_{\bar k},\Lambda)+o(t_n)\big)=o(t_n),\]
  where the last equality follows from  the fact that $\nabla g(\bar x)d_{\bar k}\in T_\Lambda(g(\bar x))$.
  It follows that $f(x_n)=f(\bar x)+t_n\nabla f(\bar x)d_{\bar k}+o(t_n)< f(\bar x)$ for all $t_n$ sufficiently small contradicting the optimality of $\bar x$ for the problem (P). Hence $d$ is a local minimizer for the problem \eqref{EqLinProgr} and the basic optimality condition \cite[Theorem 6.12]{RW98}
  \begin{equation}\label{EqBasicOpt}-\nabla f(\bar x)\in\widehat N_{\nabla g(\bar x)^{-1}\big(T_\Lambda(g(\bar x))\big)}(d)\end{equation}
  is fulfilled.

  By taking polars in  both sides of the  directional non-degeneracy condition  \eqref{EqNondegeneracy}, we have
 \begin{equation*}
 \nabla g(\bar{x})\Re^n+\big({\rm span}N_{\Lambda}(g(\bar{x});\nabla g(\bar x)d)\big)^\circ=\Re^m.
 \end{equation*}
Since
$${\rm span\;}N_{\Lambda}(g(\bar{x});\nabla g(\bar x)d)={\rm span\;}\cl\co N_{\Lambda}(g(\bar{x});\nabla g(\bar x)d)= {\rm span}N^c_{\Lambda}(g(\bar{x});\nabla g(\bar x)d)$$
and the Clarke directional normal cone is a closed convex cone,
we have
 \begin{eqnarray*}\lefteqn{\big({\rm span\;}N_{\Lambda}(g(\bar{x});\nabla g(\bar x)d)\big)^\circ}
   \\&=& \big({\rm span\;}N^c_{\Lambda}(g(\bar x);\nabla g(\bar x)d)\big)^\circ=\big(N_\Lambda ^c (g(\bar x);\nabla g(\bar x)d) -N_\Lambda ^c (g(\bar x);\nabla g(\bar x)d)\big)^\circ\\
 &=&N_\Lambda^c(g(\bar{x});\nabla g(\bar x)d)^\circ\cap \big(-N_\Lambda^c(g(\bar{x});\nabla g(\bar x)d)\big)^\circ\\
 &=& \widehat{T}_\Lambda(g(\bar{x});\nabla g(\bar x)d)\cap \big(-\widehat{T}_\Lambda(g(\bar{x});\nabla g(\bar x)d)\big)=\mathcal{L}(\widehat{T}_{\Lambda}(g(\bar{x});\nabla g(\bar x)d)),
 \end{eqnarray*}
 where the last equality follows from the fact that the directional regular tangent cone is a closed  convex cone. Hence we have shown that the directional non-degeneracy condition \eqref{EqNondegeneracy} is equivalent to
 \[\nabla g(\bar{x})\Re^n+\mathcal{L}(\widehat{T}_{\Lambda}(g(\bar{x});\nabla g(\bar x)d))=\Re^m.\]
 Note that $\mathcal{L}(\widehat{T}_{\Lambda}(g(\bar{x});\nabla g(\bar x)d))\subseteq \mathcal{L}(T_{T_{\Lambda}(g(\bar{x}))}(\nabla g(\bar x)d))$ because
 \[T_{T_{\Lambda}(g(\bar{x}))}(\nabla g(\bar x)d)+\widehat{T}_{\Lambda}(g(\bar{x});\nabla g(\bar x)d)=T_{T_{\Lambda}(g(\bar{x}))}(\nabla g(\bar x)d)\]
 by Proposition \ref{PropRegTanCone}. It follows that
\[\nabla g(\bar{x})\Re^n + \mathcal{L}(T_{T_{\Lambda}(g(\bar{x}))}(\nabla g(\bar x)d))= \Re^m\]
and, since the mapping $w\rightrightarrows \nabla g(\bar x)w-T_\Lambda(g(\bar x))$ is metrically subregular at $(d,0)$ by Lemma \ref{LemSubregDM}, we may invoke \cite[Theorem 4]{GO16-1} to obtain
$\widehat N_{\nabla g(\bar x)^{-1}\big(T_\Lambda(g(\bar x))\big)}(d)=\nabla g(\bar x)^T\widehat N_{T_\Lambda(g(\bar x))}(\nabla g(\bar x)d)$. By \eqref{EqBasicOpt}, it follows that $\emptyset\not=\Lambda^s(\bar x;d)\subseteq \Lambda(\bar x;d)\subseteq\Lambda^c(\bar x;d)$.    Since
   \[N^c_{\Lambda}(g(\bar{x});\nabla g(\bar{x})d)-N^c_{\Lambda}(g(\bar{x});\nabla g(\bar{x})d)= {\rm span\;}N^c_{\Lambda}(g(\bar{x});\nabla g(\bar{x})d)={\rm span\;}N_{\Lambda}(g(\bar{x});\nabla g(\bar{x})d),\]
   condition \eqref{EqNondegeneracy} ensures that $\Lambda^c(\bar{x};d)$ is a singleton $\{\lambda_0\}$. Hence   the claimed result $\Lambda^s(\bar x;d)= \Lambda(\bar x;d)=\Lambda^c(\bar x;d)=\{\lambda_0\}$ follows.
\end{proof}
\begin{corollary}\label{CorSpanNormal}
  Let $\bar{x}$ be a local optimal solution of problem (P). Suppose that for $d\in C(\bar{x})$  the  directional non-degeneracy condition \eqref{EqNondegeneracy} holds.
   Then $\Lambda^s(\bar x;d) = \Lambda(\bar x;d)= \Lambda^c(\bar x;d)= \{\lambda_0\}$ and the second-order condition \eqref{EqSecOrdSingleton} is fulfilled.
  \end{corollary}

The following result extends the second-order necessary optimality condition for convex set-constrained problems as in
 \cite[Proposition 3.46]{BS}  to allow the set $\Lambda$ to be  non-convex. Note that in  \cite[Proposition 3.46]{BS},  both Robinson's constraint qualification and the uniqueness of the multipliers are required while we derive our result under  {a} %the generalized LICQ/
 nondegeneracy condition which is stronger than Robinson's CQ but guarantees the uniqueness of the multipliers.
\if{
 \begin{corollary}\label{coro-LICQ}
  Let $\bar{x}$ be a local optimal solution of problem (P) and assume that the generalized LICQ
  \begin{equation}\label{com-2-3}
  \nabla g(\bar{x})^T\lambda=0, \ \ \lambda\in {\rm span}N_{\Lambda}(g(\bar{x})) \Longrightarrow \lambda=0
  \end{equation}
  or equivalently the nondegeneracy condition
  \begin{equation} \label{nondeg}  \nabla g(\bar{x})\Re^n+\mathcal{L}(\widehat{T}_{\Lambda}(g(\bar{x})))=\Re^m
  \end{equation} holds. Then $ \Lambda^c(\bar x) = \Lambda(\bar x) =
  \Lambda^F(\bar x) =
  \{\lambda_0\}$ and
%  Then there exists  a unique multiplier ${\lambda}$ satisfying the first-order optimality condition
%  \begin{equation}\label{com-2-3}
%  \nabla f(\bar{x})+\nabla g(\bar{x})^T\lambda=0, \ \ \lambda\in N_{\Lambda}(g(\bar{x}))
%  \end{equation}
 for all $d\in C(\bar{x})$ there holds
  \begin{equation*}
  %\label{com-2-5}
  \nabla^2_{xx}L(\bar{x},{\lambda_0})(d,d)-\sigma_{T^2_{\Lambda}(g(\bar{x});\nabla g(\bar{x})d))}({\lambda_0})\geq 0.
  \end{equation*}
  \end{corollary}
 \begin{proof}
{We only need to show that $\Lambda^F(\bar{x})$ is nonempty, since the remaining result follows from Corollary \ref{thm-span-normal} by taking $d=0$.} The nondegeneracy condition (\ref{com-2-3}) is obviously stronger than the NNAMCQ (\ref{NNAMCQ}) and so the MSCQ  holds at $\bar x$.
 By taking polar in the both side of the  generalized LICQ (\ref{com-2-3}), we have
 \begin{equation*}
 \nabla g(\bar{x})\Re^n+({\rm span}N_{\Lambda}(g(\bar{x})))^\circ=\Re^m.
 \end{equation*}
Since
$${\rm span}N_{\Lambda}(g(\bar{x}))={\rm span}\cl\co N_{\Lambda}(g(\bar{x}))= {\rm span}N^c_{\Lambda}(g(\bar{x}))$$
and the Clarke normal cone is a closed convex cone,
we have
 \begin{eqnarray*}({\rm span}N_{\Lambda}(g(\bar{x})))^\circ  &=& ({\rm span}N^c_{\Lambda}(g(\bar x)))^\circ=(N_\Lambda ^c (g(\bar x)) -N_\Lambda ^c (g(\bar x))^\circ\\
 &=&N_\Lambda^c(g(\bar{x}))^\circ\cap -(N_\Lambda^c(g(\bar{x})))^\circ
 = \widehat{T}_\Lambda(g(\bar{x}))\cap (-\widehat{T}_\Lambda(g(\bar{x})))=\mathcal{L}(\widehat{T}_{\Lambda}(g(\bar{x}))),
 \end{eqnarray*}
 where the last equality follows from the fact that the regular tangent cone is a closed  convex cone.
This proves the equivalence between the generalized LICQ (\ref{com-2-3}) and the nondegeneracy condition (\ref{nondeg}).
 Note that $\mathcal{L}(\widehat{T}_{\Lambda}(g(\bar{x})))\subseteq \mathcal{L}(T_{\Lambda}(g(\bar{x})))$ since
 $T_\Lambda(g(\bar{x}))+\widehat{T}_\Lambda(g(\bar{x}))=T_\Lambda(g(\bar{x}))$ by Proposition \ref{Prop2.2}. Then
 (\ref{nondeg}) implies
 \[\nabla g(\bar{x})\Re^n+\mathcal{L}(T_{\Lambda}(g(\bar{x})))=\Re^m.\]
The above condition and the {MSCQ}  together ensures that  $\widehat{N}_{\mathcal{F}}(\bar{x})=\nabla g(\bar{x})^T\widehat{N}_{\Lambda}(g(\bar{x}))$ by \cite[Theorem 4]{GO16-1}.  Hence
$$0\in\nabla f(\bar x)+ \nabla g(\bar{x})^T\widehat{N}_{\Lambda}(g(\bar{x})).$$
which implies that the multiplier set $\Lambda^F(\bar x)$ is nonempty.
 \end{proof}
}\fi
{
\begin{corollary}\label{CorNondegen}
  Let $\bar{x}$ be a local optimal solution of problem (P) and assume that the non-degeneracy condition
  \begin{equation}\label{EqGenNondegen}
  \nabla g(\bar{x})^T\lambda=0, \ \lambda\in {\rm span\;}N_{\Lambda}(g(\bar{x}))\ \Longrightarrow\ \lambda=0
  \end{equation}
  holds. Then there is a unique multiplier $\lambda_0$ satisfying the first-order optimality conditions
  \begin{equation}\label{EqS_stat}\nabla_xL(\bar x,\lambda_0)=0, \ \lambda_0\in\widehat N_\Lambda(g(\bar x)).\end{equation}
  Further, for all $d\in C(\bar x)$ we have
  \[\nabla_{xx}^2L(\bar x,\lambda_0)-\sigma_{T^2_\Lambda(g(\bar x);\nabla g(\bar x)d)}(\lambda_0)\geq 0.\]
\end{corollary}
\begin{proof}The existence and uniqueness of the multiplier $\lambda_0$ fulfilling \eqref{EqS_stat} follows from Lemma \ref{LemDirS_stat} applied with $d=0$. Further,
\[\{\lambda_0\}=\{\lambda\in\widehat N_\Lambda(g(\bar x))\mid \nabla_x L(\bar x,\lambda)=0\}\subseteq \{\lambda\in N_\Lambda(g(\bar x))\mid \nabla_x L(\bar x,\lambda)=0\}=\Lambda(\bar x)\]
and the non-degeneracy condition \eqref{EqGenNondegen} ensures that $\Lambda(\bar x)=\{\lambda_0\}$. Further note that for every $d\in C(\bar x)$ the condition \eqref{EqGenNondegen} implies \eqref{EqNondegeneracy} and $\emptyset\not=\Lambda(\bar x;d)\subseteq\Lambda(\bar x)=\{\lambda_0\}$. This shows $\Lambda(\bar x;d)=\{\lambda_0\}$ and the second statement follows from Corollary \ref{CorSpanNormal}.
\end{proof}
\begin{remark}
  According to \cite{FleKanOut07}, the first-order optimality conditions  \eqref{EqS_stat} are called {\em S-stationarity conditions}.
\end{remark}
}
\section{Second-order sufficient conditions}
We now consider sufficient conditions for optimality.  We need the following definition of an upper second order approximation set of $\Lambda$ which is a special case of the definition given in \cite[Definition 3.82]{BS}.
   \begin{definition}  Let $\bar x $ be a feasible solution of problem (P) and  $d\in C(\bar x )$.  We say that a closed set
       $\mathcal{A}(d)$ is an {\em upper second-order approximation set} for $\Lambda$ at $g(\bar x )$ in direction $\nabla g(\bar x )d\in T_\Lambda (g(\bar x ))$, if for any sequence
       $y_n \in \Lambda$ of the form $y_n:=g(\bar x )+t_n \nabla g(\bar x )d+\frac{1}{2}t_n^2(\nabla g(\bar x )w_n+a_n)$, where $t_n\downarrow 0$
       and  $\{a_n\}$ being a convergent sequence  and $\{w_n\}$ satisfying $t_nw_n\to 0$, the following condition holds
       \begin{equation*}
       \lim\limits_{n\to \infty}{\rm dist}(\nabla g(\bar x )w_n+a_n,\mathcal{A}(d))=0.
       \end{equation*}
       \end{definition}
\if{
       By definition $ T^2_\Lambda (g(\bar x );  \nabla g(\bar x )d)\subseteq  \mathcal{A}(d) $.
   Although the set $\Lambda$ may be  non-convex, when the set of multipliers $ \Lambda^F({\bar x }):=\{\lambda| 0=\nabla_x L(\bar x,\lambda), \lambda \in \widehat N_\Lambda (g(\bar x))\}$ is nonempty, we still have the expression for the critical cone $C(\bar x )=\big\{d|\, \nabla f(\bar x )d= 0, \nabla g(\bar x )d\in T_{\Lambda}(g(\bar x )) \big\}$. Using this fact, it is easy to check that  \cite[Theorem 3.83]{BS} still hold for a  non-convex set $\Lambda$.
 \begin{theorem}\label{Thm5.2}  Let $\bar x $ be a feasible point of (P). % such that  there exists a  multiplier $\lambda\in \Lambda^F(\bar x )$.
 Assume that   every $d\in C(\bar x )$ corresponds to    $\mathcal{A}(d)$, an upper second-order approximation set for $\Lambda$ at $g(\bar x )$ in direction $d$.  Further assume that for every  $d\in C(\bar x ) \setminus \{0\}$ {there is some $\lambda$ satisfying \eqref{EqS_stat} and}
\begin{equation*}
 \nabla^2_{xx}L(\bar x ,\lambda)(d,d) - %\sigma \left( \lambda| \mathcal{A}(d) \right)
{\sigma_{\mathcal{A}(d)}(\lambda)} > 0.
 \end{equation*} Then the second order growth condition holds at $\bar x $, i.e., there exists {a neighborhood $U$ of $\bar x$ and} $\delta>0$ such that
 $$f(x) \geq f(\bar x )+\delta \|x-\bar x \|^2 \quad  \forall x {\in U} \mbox{ s.t. } g(x)\in \Lambda.$$
 \end{theorem}
}\fi
Consider the so-called {\em generalized Lagrangian} $L^g:\mathbb{R}^n\times\mathbb{R}\times\mathbb{R}^m\to\mathbb{R}$ defined by
\[L^g(x,\alpha,\lambda)=\alpha f(x)+g(x)^T\lambda.\]
It is easy to check that the following variant of \cite[Theorem 3.83]{BS}  holds for a  non-convex set $\Lambda$.
 \begin{theorem}\label{Thm5.2}  Let $\bar x $ be a feasible point of (P).
 Assume that   every $d\in C(\bar x )$ corresponds to    $\mathcal{A}(d)$, an upper second-order approximation set for $\Lambda$ at $g(\bar x )$ in direction $d$.  Further assume that for every  $d\in C(\bar x ) \backslash \{0\}$  there is some $(\alpha,\lambda)\in\mathbb{R}\times\mathbb{R}^m$ satisfying
 \begin{equation}\label{EQAlpha}\alpha\geq 0,\ \alpha\nabla f(\bar x)d=0,\ \nabla_x L^g(\bar x,\alpha,\lambda)=0\end{equation}
  and
\begin{equation}\label{EqSecOrderSuff}
 \nabla^2_{xx}L^g(\bar x ,\alpha,\lambda)(d,d) - \sigma_{\mathcal{A}(d)}(\lambda) > 0.
 \end{equation}
 Then the second order growth condition holds at $\bar x $, i.e., there exists a neighborhood $U$ of $\bar x$ and $\delta>0$ such that
 $$f(x) \geq f(\bar x )+\delta \|x-\bar x \|^2 \quad  \forall x {\in U} \mbox{ s.t. } g(x)\in \Lambda.$$
 \end{theorem}
 The second-order condition \eqref{EqSecOrderSuff} has the following two implications. Firstly, {if ${\cal A}(d)\not=\emptyset$} it is easy to see that $(\alpha,\lambda)\not=(0,0)$. Secondly, we have $\lambda\in N^c_\Lambda(g(\bar x);\nabla g(\bar x)d)$ whenever $ T^2_\Lambda (g(\bar x );  \nabla g(\bar x )d)\not=\emptyset$. Indeed, by the definition we have $\emptyset\not= T^2_\Lambda (g(\bar x );  \nabla g(\bar x )d)\subseteq  \mathcal{A}(d) $ for every upper second-order approximation set and hence $\lambda\in \big(\widehat T_\Lambda(g(\bar x);\nabla g(\bar x)d)\big)^\circ=
N^c_\Lambda(g(\bar x);\nabla g(\bar x)d)$ by virtue of Propositions \ref{PropRegTanCone} and \ref{polarity}, since otherwise $$\sigma_{\mathcal{A}(d)}(\lambda)\geq \sigma_{T^2_\Lambda (g(\bar x );  \nabla g(\bar x )d)}(\lambda)=\infty.$$

 In general $ T^2_{\Lambda}(g(\bar x );
 \nabla g(\bar x )d)$ may not be an upper second-order approximation set for $\Lambda$ at $g(\bar x )$ in direction $\nabla g(\bar x )d$. But if   $ T^2_{\Lambda}(g(\bar x );
 \nabla g(\bar x )d)$ is an upper second-order approximation set for $\Lambda$ at $g(\bar x )$ in direction $\nabla g(\bar x )d$, then we say that $\Lambda$ is outer second-order regular at $g(\bar x )$ in direction $\nabla g(\bar x )d$; see \cite[Definition 3.85]{BS}.

 \if{Combining Theorem \ref{Thm5.2} and Corollary \ref{CorNondegen}, we obtain immediately the following ``no-gap'' necessary and sufficient optimality conditions under the outer second order regularity of $\Lambda$ and nondegeneracy condition. Thus \cite[Theorem 3.86]{BS} is extended to the non-convex set $\Lambda$.
 \begin{theorem} Let $\bar x $ be a feasible solution of (P) and $\Lambda$ is outer second order regular at $g(\bar x )$ in direction $\nabla g(\bar x )d$ for every  $d\in C(\bar x )$. If %there exists $\lambda \in \Lambda^F(\bar x )$ such that
 {for every $d\in C(\bar x)$ there exists $\lambda$ fulfilling \eqref{EqS_stat} and}
$$  \nabla^2_{xx}L(\bar x ,\lambda)(d,d) - %\sigma \left( \lambda_0| T^2_{\Lambda}  (g(\bar x );\nabla g(\bar x )d) \right)
{\sigma_{ T^2_{\Lambda}  (g(\bar x );\nabla g(\bar x )d)}}(\lambda)>0,$$
 then the second order growth condition holds at point $\bar x $. If, in addition, the nondegeneracy condition  \eqref{EqNondegeneracy} holds at $\bar x $ and $ T^2_{\Lambda}(g(\bar x );
 \nabla g(\bar x )d)$ is convex for any $ d\in C(\bar x )$, then the second order conditions
 $$  \nabla^2_{xx}L(\bar x ,\lambda_0)(d,d) - %\sigma \left( \lambda_0| T^2_{\Lambda} (g(\bar x );\nabla g(\bar x )d) \right)
 {\sigma_{T^2_{\Lambda} (g(\bar x );\nabla g(\bar x )d)}(\lambda_0)}>0, \forall d \in C(\bar x ) \setminus \{0\}$$
 are necessary and sufficient for the second order growth condition at the point $\bar x $, {where $\lambda_0$ denotes the unique multiplier fulfilling \eqref{EqS_stat}.}
 \end{theorem}
 }\fi
Combining theorems  \ref{second-order-theorem-1} and \ref{Thm5.2}, we obtain immediately the following ``no-gap'' necessary and sufficient optimality conditions under the outer second order regularity of $\Lambda$. Thus \cite[Theorem 3.86]{BS} is extended to the non-convex set $\Lambda$ and Robinson' constraint qualification is weakened to directional metric subregularity.
 \begin{theorem} Let $\bar x $ be a feasible solution of (P) and assume that $\Lambda$ is outer second-order regular at $g(\bar x )$ in direction $\nabla g(\bar x )d$ for every  $d\in C(\bar x )\backslash \{0\}$. If for every  $d\in C(\bar x ) \backslash \{0\}$  there is some $(\alpha,\lambda)\in\mathbb{R}\times\mathbb{R}^m$ satisfying \eqref{EQAlpha}
  and
\begin{equation}\label{EqSecOrderSuffReg}
 \nabla^2_{xx}L^g(\bar x ,\alpha,\lambda)(d,d) - \sigma_{T^2_{\Lambda}  (g(\bar x );\nabla g(\bar x )d)}(\lambda) > 0,
 \end{equation}
 then the second order growth condition holds at $\bar x $.
 If, in addition,  the feasible mapping $M(x)=g(x)-\Lambda$ is metrically subregular at $(\bar x,0)$ in direction $d$, $\nabla f(\bar x)d=0$  and $ T^2_{\Lambda}(g(\bar x );
 \nabla g(\bar x )d)$ is convex for every  $ d\in C(\bar x )\backslash \{0\}$, then the second order conditions
 \begin{equation}\label{EqNoGapQuadrGr}  \sup_{\lambda\in\Lambda(\bar x;d)}\Big(\nabla^2_{xx}L(\bar x ,\lambda)(d,d) -
 {\sigma_{T^2_{\Lambda} (g(\bar x );\nabla g(\bar x )d)}(\lambda)}\Big)>0,\, \forall d \in C(\bar x ) \backslash \{0\}
 \end{equation}
 are necessary and sufficient for the second order growth condition at the point $\bar x $.
 \end{theorem}
 \begin{proof}The sufficiency of \eqref{EqNoGapQuadrGr} for the quadratic growth condition follows from \eqref{EqSecOrderSuffReg} by taking $\alpha=1$. There remains to show  the necessity of \eqref{EqNoGapQuadrGr} for the quadratic growth condition. Assume that $f(x)\geq f(\bar x)+\delta\norm{x-\bar x}^2$ holds for all feasible $x$ sufficiently close to $\bar x$ for some $\delta>0$ and consider $d\in C(\bar x)\backslash \{0\}$. Then $\Lambda(\bar x;d)\not=\emptyset$ by Proposition \ref{PropDirFirstOrder} and $\bar x$ is a local minimizer of the problem
  \[\min_x f(x)- \delta\norm{x-\bar x}^2\quad\mbox{subject to}\quad g(x)\in\Lambda.\]
 By Theorem \ref{second-order-theorem-1}, there is some $\lambda\in \Lambda(\bar x;d)$ such that
 \[\nabla_{xx}^2 L(\bar x,\lambda)(d,d)-2\delta\norm{d}^2- \hat\sigma_{T^2_{\Lambda} (g(\bar x );\nabla g(\bar x )d)}(\lambda)\geq 0.\]
 But by assumption $ T^2_{\Lambda}(g(\bar x );
 \nabla g(\bar x )d)$ is convex and hence $$\hat\sigma_{T^2_{\Lambda} (g(\bar x );\nabla g(\bar x )d)}(\lambda)=\sigma_{T^2_{\Lambda} (g(\bar x );\nabla g(\bar x )d)}(\lambda)$$ by Proposition \ref{PropHatSigma}(2),
  and \eqref{EqNoGapQuadrGr} follows, provided $T^2_{\Lambda} (g(\bar x );\nabla g(\bar x )d)\not=\emptyset$.
On the other hand, if $T^2_{\Lambda} (g(\bar x );\nabla g(\bar x )d)=\emptyset$ then \eqref{EqNoGapQuadrGr} automatically holds because the support function of the empty set is identical $-\infty$ by definition.
 \end{proof}
\section{Examples}
{In this section we use some examples to illustrate our theory.} We will apply our theory to the class of SOC-MPCCs and MPCCs in a forthcoming paper \cite{GYZZ}.
\begin{example}
  Consider the one-dimensional problem
  \[\min_{x\in\R}-\frac 12 x^2\ \ {\rm s.t.}\ \ g(x):=(x^2,x)\in \Lambda:=C_1\cup C_2\]
  at the reference point $\bar x=0$,   where $C_1:=\{(x_1,x_2)\, |\, (x_1-1)^2+x_2^2\leq 1\}$ and $C_2:=\{(x_1,x_2)\,|\, (x_1+1)^2+x_2^2\leq 1\}$ are unit circles with center $(1,0)$ and $(-1,0)$, respectively. Clearly, the feasible region is ${\cal F}=[-1,1]$ and thus $\bar x$ is not a local minimizer. We want to check whether we can reject $\bar x$ as a local minimizer by our theory. First at all note that the feasible set mapping $g(x)-\Lambda$ is metrically subregular at $(\bar x,0)$ because all $x\in\R$ sufficiently close to $\bar x$  are feasible. Straightforward calculations yield
  \[T_\Lambda(g(\bar x))=T_{C_1}(g(\bar x))\cup T_{C_2}(g(\bar x)) =(\R_+\times\R) \cup (\R_-\times\R)=\R^2,\]
   $\widehat N_\Lambda(g(\bar x))=\{(0,0)\}$, and $C(\bar x)=\R$. Consider the  critical direction $d=1$, the discussion for the opposite direction $d=-1$ could be performed similarly. Utilizing \cite[Propositions 3.30 and 3.37]{BS} it is not difficult to show that
   \begin{eqnarray*}T^2_\Lambda(g(\bar x);\nabla g(\bar x)d)&=&T^2_{C_1}(g(\bar x);\nabla g(\bar x)d)\cup T^2_{C_2}(g(\bar x);\nabla g(\bar x)d)\\
   &=&(\{t\;\mid\;t\geq 1\}\times\R)\cup(\{t\;\mid\;t\leq -1\}\times\R).
   \end{eqnarray*}
   Since $\nabla g(\bar x)0+\nabla ^2g(\bar x,\bar x)(d,d)=(2,0)\in T_\Lambda^2(g(\bar x);\nabla g(\bar x)d)$, we obtain $0\in T^2_{\cal F}(\bar x;d)$ from \eqref{EqSecOrdTanSetF}. By observing $\nabla f(\bar x)0+\nabla^2f(\bar x)(d,d)=-1<0$, we may conclude from Theorem \ref{ThBasicSecOrd} that $\bar x$ is not a local minimizer.

   Next let us apply Theorem \ref{second-order-theorem-1}.
   Since for all sufficiently small $x>0$  the point $y^1(x)=(1-\sqrt{1-x^2},x)$ belongs to $C_1$ but not to $C_2$, we obtain $\widehat N_\Lambda(y^1(x))= \widehat N_{C_1}(y^1(x))=\R_+(-\sqrt{1-x^2}, x)$. Together with $\lim_{x\downarrow0}\frac{y^1(x)-(0,0)}x=(0,1)=\nabla g(\bar x)d$ we obtain $\R_+(-1,0)=\R_-\times\{0\}\subseteq N_\Lambda(g(\bar x);\nabla g(\bar x)d)$. Similar arguments using the points $y^2(x)=(\sqrt{1-x^2}-1,x)$ show $\R_+\times\{0\}\subseteq N_\Lambda(g(\bar x);\nabla g(\bar x)d)$ and in fact, it follows that $N_\Lambda(g(\bar x);\nabla g(\bar x)d)=\R\times\{0\}$. It follows that for every $\lambda\in  N_\Lambda(g(\bar x);\nabla g(\bar x)d)$ we have $\nabla g(\bar x)^T\lambda=0$ and consequently $\Lambda(\bar x;d)=N_\Lambda(g(\bar x);\nabla g(\bar x)d)=\R\times\{0\}$. By taking
   \[A=\nabla g(\bar x)\R+\nabla^2g(\bar x)(d,d)=\{2\}\times\R\subset {\rm int\;} T^2_\Lambda(g(\bar x);\nabla g(\bar x)d)\]
   we have
   \[\hat \sigma_{T^2_\Lambda(g(\bar x);\nabla g(\bar x)d), A}(\lambda)=\begin{cases}0&\mbox{if $\lambda=(0,0)$,}\\
   \infty &\mbox{else,}
   \end{cases}\]
   because  $u\in N^{-1}_{T^2_\Lambda(g(\bar x);\nabla g(\bar x)d)}(\lambda)\cap A=\emptyset$ whenever $\lambda\not=(0,0)$. Thus  we obtain
   \[\nabla_{xx}^2L(\bar x,\lambda)-\hat \sigma_{T^2_\Lambda(g(\bar x);\nabla g(\bar x)d), A}(\lambda)=\begin{cases}-1&\mbox{if $\lambda=(0,0)$,}\\
   -\infty &\mbox{else}
   \end{cases}\]
   violating condition \eqref{EqSecOrderOpt}. Thus we can also conclude from Theorem \ref{second-order-theorem-1} that $\bar x$ is not a local minimizer.

   On the other hand, for $\lambda:=(1,0)\in \Lambda(\bar x;d)$ we have $N^{-1}_{T^2_\Lambda(g(\bar x);\nabla g(\bar x)d)}(\lambda)=\{-1\}\times\R$ and $\hat \sigma_{T^2_\Lambda(g(\bar x);\nabla g(\bar x)d)}(\lambda)=-1$ follows. Hence
   \[\nabla_{xx}^2L(\bar x,\lambda)-\hat \sigma_{T^2_\Lambda(g(\bar x);\nabla g(\bar x)d)}(\lambda)=-1-(-1)=0\]
   and the second-order necessary condition \eqref{EqSecOrderOptnew} is fulfilled.

   Note that we can not apply Theorem \ref{CorClSecondOrder} because condition \eqref{EqClMetrReg} fails to hold because of $\nabla g(\bar x)^T\lambda=0$ with $\lambda=(1,0)\in N_\Lambda(g(\bar x);\nabla g(\bar x)d)$.
\end{example}

\begin{example}\label{ExMPCC}
  Consider the {MPCC } %program
  \begin{eqnarray*}\min && f(x_1,x_2,x_3):=x_1-x_2+x_3 +\frac 12 x_2^2- x_3^2\\
  {\rm s.t.}&& g(x_1,x_2,x_3):=\left(\begin{array}{c}-x_1\\-x_2\\-4x_1+x_2-x_3\\ -3x_2-x_3\end{array}\right)\in\Lambda:=D_{CC}\times\R_-\times\R_-,
  \end{eqnarray*}
  at $\bar x=(0,0,0)$, where $D_{CC}:=\{(a,b)\in\R^2_-\,|\, ab=0\}$ denotes {the complementarity cone in $\R^2$}.
  %the (negative) {\em complementarity angle}.
   By considering points of the form $(0,t,t)$ with $t>0$ it is easy to see that $\bar x$ is not a local minimizer and we want to verify this with our theory. We claim that $C(\bar x)=\{(0,t,t)\,|\,t\geq 0\}$. Indeed, the inclusion ``$\supseteq$'' obviously holds and we only have to verify the opposite  inclusion. Since $g$ is linear and $\Lambda$ is a cone, every critical direction $d$ must fulfill $g(d)\in\Lambda$, from which, by considering the third and fourth component of the system, the inequality
  \[d_2-d_3\leq \min\{4d_1,4d_2\}=0\]
  follows. Together with $\nabla f(\bar x)d=d_1-(d_2-d_3)\leq 0$ and $d_1\geq 0$ we conclude $d_1=0$ and $d_3=d_2\geq 0$ proving our claim. Now consider a critical direction $d=(0,t,t)$ with $t>0$. By using \cite[Lemma 4.1]{Gfr14a} we obtain
  \[N_\Lambda(g(\bar x); \nabla g(\bar x) d)=\R\times\{0\}\times\R_+\times\{0\}=N_\Lambda^c(g(\bar x); \nabla g(\bar x) d)\]
  and thus  both {DirRCQ} %condition
   \eqref{EqClMetrReg} and the directional non-degeneracy condition \eqref{EqNondegeneracy} are fulfilled. Straightforward calculation yield that
  $\Lambda(\bar x;d)=\Lambda^c(\bar x;d)=\Lambda^s(\bar x;d)=\{\lambda_0\}$ with $\lambda_0=(-3,0,1,0)$. Further, $T^2_\Lambda(g(\bar x);\nabla g(\bar x)d)=\{0\}\times\R\times\R_-\times\R$ and $\sigma_{T^2_\Lambda(g(\bar x);\nabla g(\bar x)d)}(\lambda_0)=0$. Hence
  \[\nabla_{xx}^2L(\bar x,\lambda_0)(d,d)-\sigma_{T^2_\Lambda(g(\bar x);\nabla g(\bar x)d)}(\lambda_0)=t^2-2t^2=-t^2<0\]
  and we can reject $\bar x$ as a local minimizer by means of Corollary \ref{CorSpanNormal}. Further we could reject $\bar x$ also by Theorem \ref{second-order-theorem-1} together with Proposition \ref{PropHatSigma}(2) due to  convexity of the second-order tangent set $T^2_\Lambda(g(\bar x);\nabla g(\bar x)d)$. However, we cannot apply Corollary \ref{CorNondegen} because condition \eqref{EqGenNondegen} is not fulfilled. {This example demonstrates that directional non-degeneracy condition \eqref{EqNondegeneracy}
{is} strictly weaker than the non-directional condition \eqref{EqGenNondegen}, which, in terms of the literature on MPCC, is equivalent to MPCC-LICQ.}
\end{example}

\begin{example}
  Consider the problem
    \begin{eqnarray*}\min && f(x_1,x_2,x_3):=x_1-x_2+x_3 +\frac 12 x_2^2- x_3^2\\
  {\rm s.t.} && g(x_1,x_2,x_3):=\left(\begin{array}{c}-x_1\\-x_2\\-4x_1+x_2-x_3+x_2^2\\ -3x_2-x_3\end{array}\right)\in\Lambda:=D_{CC}\times\R_-\times\R_-,
  \end{eqnarray*}
  at $\bar x=(0,0,0)$, which differs from the problem in the preceding example only by the presence of the term $x_2^2$ in the third component of $g$. Thus $C(\bar x)$ remains unchanged and it is easy to see that $\Lambda$ is outer second-order regular at $g(\bar x)$ in direction $\nabla g(\bar x)d$ for every critical direction $d=(0,t,t)$ with $t>0$. Further,
  \[\nabla_{xx}^2L(\bar x,\lambda_0)(d,d)-\sigma_{T^2_\Lambda(g(\bar x);\nabla g(\bar x)d)}(\lambda_0)=t^2>0\]
  and therefore the sufficient second-order condition  \eqref{EqSecOrderSuffReg}  is fulfilled implying that $\bar x$ is a strictly local minimizer satisfying the second-order growth condition.
\end{example}

\begin{example}
  Consider the program
  \begin{eqnarray*}\min&& f(x_1,x_2)=-x_1+\frac 12 x_2^2\\
  {\rm s.t.} &&g(x_1,x_2)=\left(\begin{array}
    {c}-x_1\\-x_2\\x_1^2-x_2
  \end{array}\right)\in\Lambda:=D_{cc}\times\R_-
  \end{eqnarray*}
  at $\bar x=(0,0)$. Then $C(\bar x)=(\{0\}\times\R_+)\cup (\R_+\times\{0\})$ and $\Lambda$ is outer second-order regular at $g(\bar x)$ in direction $\nabla g(\bar x)d$ for every critical direction $d\in C(\bar x)$. If $d=(t,0)$, $t>0$, then $N_\Lambda(g(\bar x);\nabla g(\bar x)d)=\{0\}\times\R\times\R_+$ and $T^2_\Lambda(g(\bar x);\nabla g(\bar x)d)=\R\times\{0\}\times\R_-$ and condition \eqref{EqSecOrderSuffReg} is fulfilled with $\alpha=0$ and $\lambda=(0,-1,1)$. On the other hand, for $d=(0,t)$, $t>0$, we have $N_\Lambda(g(\bar x);\nabla g(\bar x)d)=\R\times\{0\}\times\{0\}$ and $T^2_\Lambda(g(\bar x);\nabla g(\bar x)d)=\{0\}\times\R\times\R$. Now \eqref{EqSecOrderSuffReg} is fulfilled with $\alpha=1$, $\lambda=(-1,0,0)$
  and we have verified that $\bar x$ is a strictly local minimizer fulfilling the second-order growth condition.
\end{example}

 \end{document}